\documentclass[12pt]{amsart}
\usepackage[margin=1in]{geometry}
\usepackage[utf8]{inputenc}
\usepackage[usenames, dvipsnames]{color}
\usepackage{pb-diagram}
\usepackage{hyperref}
\usepackage{bm}
\usepackage{amssymb}
\usepackage[all]{xy}
\usepackage{tikz-cd}
\tikzset{shorten <>/.style={shorten >=#1,shorten <=#1}}
\usepackage{braids}
\usepackage{placeins}
\usepackage{lscape}
\usepackage{dsfont}
\usepackage{graphics}
\usepackage{enumitem}

\newcommand{\Inv}{\operatorname{Inv}}

\newcommand{\cC}{{\mathcal C}}
\newcommand{\cZ}{{\mathcal Z}}

\newcommand{\ot}{{\otimes}}

\newcommand{\ca}{{\mathcal{C} }}

\newcommand{\TY}{{\mathcal{TY}}}

\newcommand{\Ker}{\mbox{\rm Ker\,}}

\newcommand{\ku}{{\Bbbk}}
\newcommand{\Z}{{\mathbb Z}}

\newcommand{\id}{\operatorname{id}}

\newcommand{\cB}{\mathcal{B}}

\newcommand{\cD}{\mathcal{D}}

\newcommand{\Vect}{\mbox{\rm Vect\,}}

\newcommand{\Fun}{\operatorname{Fun}}

\newcommand{\Rep}{\operatorname{Rep}}
\newcommand{\Irr}{\operatorname{Irr}}

\newcommand{\Hom}{\operatorname{Hom}}

\newcommand{\End}{\operatorname{End}}

\newcommand{\Id}{\operatorname{Id}}

\newcommand{\Tr}{\operatorname{Tr}}

\newcommand{\Pic}[1]{\operatorname{Pic}(#1)}

\newcommand{\Aut}{\operatorname{Aut}}

\theoremstyle{plain}
\numberwithin{equation}{section}

\newtheorem{theorem}{Theorem}[section]
\newtheorem{lemma}[theorem]{Lemma}

\newtheorem{corollary}[theorem]{Corollary}
\newtheorem{proposition}[theorem]{Proposition}

\theoremstyle{definition}
\newtheorem{definition}[theorem]{Definition}
\newtheorem{example}[theorem]{Example}
\theoremstyle{remark}

\newtheorem{remark}[theorem]{Remark}
\theoremstyle{remark}

\newcommand{\parcen}[1]{\phantom{(}#1\phantom{)}}

\newcommand{\br}{to [out=90,in=-90]}

\newcommand{\onebox}[3]{
\begin{scope}[line width=1]
\draw (0,0) node[below] {$#1$} --(0,2) node[above] {$#2$};
\draw[fill=white] (-.33,.66) rectangle node {$#3$} (.33,1.33);
\end{scope}
}

\newcommand{\onecirc}[3]{
\begin{scope}[line width=1]
\draw (0,0) node[below] {$#1$} --(0,2) node[above] {$#2$};
\draw[fill=white] (0,1) circle (.45) node {$#3$};
\end{scope}
}

\newcommand{\twobox}[5]{
\begin{scope}[line width=1]
\draw (0,0) node[below] {$#1$} --(0,2) node[above] {$#2$};
\draw (1,0) node[below] {$#3$} --(1,2) node[above] {$#4$};
\draw[fill=white] (-.33,.66) rectangle node {\small $#5$} (1.33,1.33);
\end{scope}
}

\newcommand{\twoboxnotop}[3]{
\begin{scope}[line width=1]
\draw (0,0) node[below] {$#1$} --(0,1);
\draw (1,0) node[below] {$#2$} --(1,1);
\draw[fill=white] (-.33,.66) rectangle node {\small$#3$} (1.33,1.33);
\end{scope}
}

\newcommand{\twoboxnobottom}[3]{
\begin{scope}[line width=1]
\draw (0,1) --(0,2) node[above] {$#1$};
\draw (1,1) --(1,2) node[above] {$#2$};
\draw[fill=white] (-.33,.66) rectangle node {\small $#3$} (1.33,1.33);
\end{scope}
}
\newcommand{\twoboxnostrands}[1]{
\begin{scope}[line width=1]
\draw[fill=white] (-.33,.66) rectangle node {\small $#1$} (1.33,1.33);
\end{scope}
}

\newcommand{\gcocyc}[4]{
 \begin{tikzpicture}[line width=1, scale=1.25]
  \foreach \x in {0,1.5}{
  \draw(\x,0)--(\x,2);
  }
  \draw[fill=white] (-.25,.75) rectangle (1.75,1.25); 
  \draw (0,0) node[below] {$\lambda(#2,#3)$};
  \draw (1.5,0) node[below] {$\lambda(#1,#2#3)$};
  \draw (0,2) node[above] {$\lambda(#1,#2)^{g_3}$};
  \draw (1.5,2) node[above] {$\lambda(#1#2,#3)$};
  \draw (.75,1) node {$#4_{#1,#2,#3}$};
  \end{tikzpicture}}

\newcommand{\assocrev}[4]{
\begin{tikzpicture}[baseline=75*#4,line width=1, scale=#4]

 \foreach \x in {0,2,4}{
  \draw(\x,0)--(\x,2);
  \draw(\x,5)--(\x,5.5);
  }
  \foreach \y in {-2,-4}
 {
 \draw(\y,0)--(\y,5.5);
 }
  
   \begin{scope}[scale=2,xshift=-1cm,yshift=2.5cm]
  \braid[number of strands=3] a_1;
  \end{scope}
  
   \draw[fill=white] (1.25,1.25) rectangle node {$#1,#2,#3$} (4.5,2.5);
   
     \draw (-4,0) node[below] {$\parcen{#1}$};
     \draw (-2,0) node[below] {$\parcen{#2}$};
   \draw (0,0) node[below] {$\parcen{#3}$};
 \draw (2,0) node[below] {$(#2,#3)$};
\draw (4.25,0) node[below] {$(#1,#2#3)$};

   \draw (-4,5.5) node[above] {$\parcen{#1}$};
     \draw (-2,5.5) node[above] {$\parcen{#2}$};
   \draw (0,5.5) node[above] {$(#1,#2)$};
 \draw (2,5.5) node[above] {$\parcen{#3}$};
\draw (4.25,5.5) node[above] {$(#1#2,#3)$};

  \end{tikzpicture}
  }
  
\newcommand{\assocfull}[5]{
\begin{tikzpicture}[baseline=75*#5,line width=1, scale=#5]
 \foreach \x in {0,2,4}{
  \draw(\x,0)--(\x,2);
  \draw(\x,5)--(\x,5.5);
  }
  \foreach \y in {-2,-4}
 {
 \draw(\y,0)--(\y,5.5);
 }
   \begin{scope}[scale=2,xshift=-1cm,yshift=2.5cm]
  \braid[number of strands=3] a_1^{-1};
  \end{scope}
   \draw[fill=white] (1.25,1.25) rectangle node {$#4_{#1,#2,#3}$} (4.5,2.5);
     \draw (-4,0) node[below] {$\parcen{X_{#1}}$};
     \draw (-2,0) node[below] {$\parcen{Y_{#2}}$};
   \draw (0,0) node[below] {$\parcen{Z_{#3}}$};
 \draw (2,0) node[below] {$\lambda(#2,#3)$};
\draw (4.75,0) node[below] {$\lambda(#1,#2#3)$};
   \draw (-4,5.5) node[above] {$\parcen{X_{#1}}$};
     \draw (-2,5.5) node[above] {$\parcen{Y_{#2}}$};
   \draw (0,5.5) node[above] {$\lambda(#1,#2)$};
 \draw (2,5.5) node[above] {$\parcen{Z_{#3}}$};
\draw (4.75,5.5) node[above] {$\lambda(#1#2,#3)$};
  \end{tikzpicture}
  }

\newcommand{\assocrevfull}[5]{
\begin{tikzpicture}[baseline=75*#5,line width=1, scale=#5]
 \foreach \x in {0,2,4}{
  \draw(\x,0)--(\x,2);
  \draw(\x,5)--(\x,5.5);
  }
  \foreach \y in {-2,-4}
 {
 \draw(\y,0)--(\y,5.5);
 }
\begin{scope}[scale=2,xshift=-1cm,yshift=2.5cm]
\braid[number of strands=3] a_1;
\end{scope}
\draw[fill=white] (1.25,1.25) rectangle node {$#4_{#1,#2,#3}$} (4.5,2.5);
\draw (-4,0) node[below] {$\parcen{X_{#1}}$};
\draw (-2,0) node[below] {$\parcen{Y_{#2}}$};
\draw (0,0) node[below] {$\parcen{Z_{#3}}$};
\draw (2,0) node[below] {$\lambda(#2,#3)$};
\draw (4.75,0) node[below] {$\lambda(#1,#2#3)$};
\draw (-4,5.5) node[above] {$\parcen{X_{#1}}$};
\draw (-2,5.5) node[above] {$\parcen{Y_{#2}}$};
\draw (0,5.5) node[above] {$\lambda(#1,#2)$};
\draw (2,5.5) node[above] {$\parcen{Z_{#3}}$};
\draw (4.75,5.5) node[above] {$\lambda(#1#2,#3)$};
\end{tikzpicture}
}
\newcommand{\zestedcap}{
\draw[looseness=2] (0,0) to [out=90, in=90] (1,0);
\draw[looseness=2] (1.5,0) to [out=90, in=90] (2.5,0)  ;
}

\author[C. Delaney]{Colleen Delaney}
\address{Department of Mathematics, UC Berkeley}
\email{cdelaney@math.berkeley.edu}
\author[C. Galindo]{C\'esar Galindo}
\address{ Departamento de Matem\'aticas, Universidad de los Andes, Bogot\'a, Colombia}
\email{cn.galindo1116@uniandes.edu.co}
\author[J. Plavnik]{Julia Plavnik}
\address{Department of Mathematics, Indiana University}
\email{jplavnik@iu.edu}
\author[E. Rowell]{Eric C. Rowell}
\address{Department of Mathematics, Texas A\&M University}
\email{rowell@math.tamu.edu}
\author[Q. Zhang]{Qing Zhang}
\address{Department of Mathematics, Purdue University}
\email{zhan4169@purdue.edu}
\begin{document}
\title{$G$-crossed braided zesting}
\begin{abstract} For a finite group $G$, a $G$-crossed braided fusion category is $G$-graded fusion category with additional structures, namely a $G$-action and a $G$-braiding.  
We develop the notion of $G$-crossed braided zesting: an explicit method for constructing new $G$-crossed braided fusion categories from a given one by means of cohomological data associated with the invertible objects in the category and grading group $G$.  This is achieved by adapting a similar construction for (braided) fusion categories recently described by the authors. All $G$-crossed braided zestings of a given category $\cC$ are $G$-extensions of their trivial component and can be interpreted in terms of the homotopy-based description of Etingof, Nikshych and Ostrik. In particular, we explicitly describe which $G$-extensions correspond to $G$-crossed braided zestings. 
\end{abstract}
\thanks{The research of E.R. was partially supported NSF grants DMS-1664359, DMS-2000331 and DMS-2205962. The research of J.P. was partially supported by NSF grants DMS-1917319 and DMS-2146392. J.P. was partially supported by Simons Foundation Award 889000 as part of the Simons Collaboration on Global Categorical
Symmetries. C.D. was supported by NSF Mathematical Sciences Postdoctoral Research Award 2002222 and the Simons Collaboration on Global Categorical Symmetries. C.G. was supported by the School of Science of Universidad de los Andes, grant INV-2020-105-2074.}
\maketitle
\section{Introduction}
Methods for constructing fusion categories from a given category $\cC$ with some kind of finite group symmetry abound: $G$-equivariantization \cite{EGNObook}, gauging \cite{BBCW,CGPW}, $G$-extensions \cite{ENO}, $G$-de-equivariantizations \cite{MugerGalois} and, more recently, zesting \cite{DGPRZ}. These constructions have a wide range of applications, including fusion rule categorification \cite{LPR}, classification problems \cite{NRWW}, topological invariants \cite{Cui}, models for topological phases of matter \cite{BBCW} and many more. In this article we add to this collection of tools by adapting the notion of zesting to $G$-crossed braided fusion categories.

One application of our results is the explicit construction of many non-trivial examples of $G$-crossed braided fusion categories.  These, in turn, have numerous important applications.  In condensed matter physics they model symmetry enriched topological phases of matter \cite{BBCW} when the trivial component is modular or super-modular: here the non-trivial components model defects. In topology, $G$-crossed braided fusion categories can be used as input for constructing $(3+1)$-dimensional topological quantum field theories \cite{Cui}. 
In the case the trivial component is modular these also yield $(2+1)$-dimensional homotopy quantum field theories \cite{TuraevHomotopyBook}. Moreover $G$-crossed fusion categories yield fusion $2$-categories (see \cite{Cui,DR}), thus elevating them as important examples in higher algebra.

Braided and associative zesting were formalized in \cite{DGPRZ} as procedures for constructing new fusion categories from a given  $G$-graded fusion category $\cC$. In most cases one further assumes that $\cC$ is braided so that the grading group is an abelian group $A$.  The first step of associative zesting is to twist the tensor product component-wise by a $2$-cocycle $\lambda\in Z^2(A,\Inv(\cC_e))$, where the action of $A$ is trivial.  This new tensor product can be made associative provided the image of $\lambda$ under the Pontryagin-Whitehead function is trivial.
In this case one adjusts the associativity constraint component-wise by a 3-cochain $\nu\in C^3(A,\Bbbk^\times)$ that depends on $\lambda$, and all such $3$-cochains form a torsor over $H^3(A,\Bbbk^\times)$.  Thus one obtains a new fusion category $\cC^{(\lambda,\nu)}$ with the same objects and grading as $\cC$ but potentially different fusion rules and associativity constraints. One may go one step further to determine if $\cC^{(\lambda,\nu)}$ can be equipped with a braiding by adjusting the braiding on $\cC$ by an isomorphism $t:\lambda(a,b)\rightarrow\lambda(b,a)$ and an additional $1$-cochain $j\in C^1(A,\Aut_\ot(\Id_\cC))$.\footnote{In most common situations $j$ may be taken to be trivial \cite[Corollary 4.16]{DGPRZ}.}  If so, one has a braided zesting $\cC^{(\lambda,\nu,t,j)}$, see Definition \ref{def:braidedzesting} for details.  

In this work we generalize zesting to $G$-crossed braided categories $\cB$.  This requires a number of adjustments, as the grading group $G$ is no longer assumed to be abelian, the $2$-cocycle $\lambda\in Z^2(G,\Inv(\cB_e))$ may have non-trivial $G$ action, and the braiding is now a $G$-braiding. 
Although $G$-crossed zesting is a generalization of associative (and braided) zesting, it is simpler in a couple of ways.  Firstly, the final step of adjusting the braiding is unnecessary--one always obtains a $G$-braiding. In particular, the pair $(\lambda,\nu)$ is all the data necessary to specify the $G$-crossed zesting.  Secondly, the appropriate notion of equivalence is somewhat more flexible in the $G$-crossed setting.  We prove that our definition is philosophically well-motivated in that two $G$-crossed braided fusion categories are equivalent if their $G$-equivariantizations are equivalent as braided fusion categories.  Finally, since any $A$-graded braided fusion category is an $A$-crossed braided category with trivial $A$-action, we obtain many non-trivial examples of $A$-crossed braided categories via associative zesting of braided fusion categories. 
We illustrate our results with several examples.  

During the preparation of this manuscript the authors became aware of related work formulating zesting for skeletal fusion categories \cite{BarkeshliCheng2020}. The work here and in \cite{DGPRZ} was carried out independently and can be used to rederive the formulas in \cite{BarkeshliCheng2020}. See also \cite{ABK}, in which similar ideas are employed as in this paper based on the work in \cite{DGPRZ}.

Here is a more detailed outline of the contents of our paper. In Section \ref{prelim} we provide the key definitions, illustrated with familiar examples.  The new $G$-crossed zesting construction is described in Section \ref{sec:Gcrossedzesting} again with a number of examples.  We also describe how the rigidity property and pivotal structures on $\cC$ are transferred to the $G$-crossed zesting $\cC^{(\lambda,\nu)}$. Finally, in Section \ref{sec:appsandexes} we discuss the $G$-crossed modular data and how it transforms under $G$-crossed zesting. Examples are provided to illustrate. 

\subsection*{Acknowledgments} The authors gratefully acknowledge the support of the American Institute of Mathematics, where
this collaboration was initiated. Part of this
research was carried out while the authors participated in a thematic month program at
CRM, which is partially supported by NSF grant DMS-2228888. The authors benefited from conversations with Parsa Bonderson, David Aasen, Richard Ng, and Zhenghan Wang. 

\section{Preliminaries}\label{prelim}
Let $\mathds{k}$ be an algebraically closed field of characteristic zero. A $\mathds{k}$-linear abelian rigid semisimple monoidal category $(\cC, \otimes, \mathds{1}, \alpha)$ with finite isomorphism classes of simple objects is a fusion category if the tensor product $\otimes$ is bilinear on morphisms and $\End(\mathds{1}) \cong \mathds{k}$. We take the convention $\alpha_{X,Y,Z}: (X \otimes Y) \otimes Z \to X \otimes (Y \otimes Z)$ for the monoidal structure of $\cC$ and refer to the isomorphisms $\alpha$ as the {\it associators}.

A $G$-grading on a tensor category $\cC$ is a decomposition $\cC=\bigoplus_{g \in G} \cC_g$ into full abelian subcategories $\cC_g \subset \cC$ that respects the tensor product, {\it i.e.}, if $X_g \in \cC_g$ and $Y_h \in \cC_h$ then $X_g \otimes Y_h \in \cC_{gh}$. We will denote the identity element of $G$ by $e$ and refer to $\cC_e$ as the trivial component of $\cC$. We will work with only those categories which are faithfully graded, so that $\cC_g \ne 0$ for all $g\in G$. Every fusion category is faithfully graded by its {\it universal grading group}, see \cite{EGNObook} for more details.

In addition to organizing the monoidal structure of $\cC$, $G$ will act on $\cC$, in which case it is important to think of $G$ as a category $\underline{G}$ with objects $g \in G$ and morphisms $\id_g$ with a monoidal structure induced by the group multiplication. Elsewhere we will use underlines to signify the category level at which we are meant to understand an object. For example, $\underline{\Aut_{\otimes}(\cC)}$ is the category of tensor autoequivalences of $\cC$ and their tensor natural transformations while $\Aut_{\otimes}(\cC)$ is the group of isomorphism classes of tensor autoequivalences.

\subsection{\texorpdfstring{$G$}{G}-crossed braided tensor categories}

\begin{definition}
\label{def:gcrossedbraided}
A {\it right} $G$-crossed braided category is a fusion category $\cC$ with
\begin{enumerate}
\item  a $G$-grading $\cC = \oplus_{g \in G}\, \cC_g$, 
\item a {\it right} $G$-action,  that is, a monoidal functor $T: \underline{G^{op}} \to \underline{\Aut_{\otimes}(\cC)}$ such that $T_h(\cC_g) \subset \cC_{h^{-1}gh}$ for all $g,h\in G$,  
\item a $G$-braiding consisting of natural isomorphisms $c_{X,Y_h}: X \otimes Y_h \to Y_h \otimes T_h(X)$ for all $X\in \cC$, $Y_h\in \cC_h$, satisfying the right action analogues of the axioms in \cite[Definition 8.24.1 (a)-(c)]{EGNObook} (see below). 
\end{enumerate}     
\end{definition}
  
We will use the notation $X^g:=T_g(X)$ for the (right) $G$-action on objects and $f^g:=T_g(f)$ for the $G$-action on morphisms. We will denote the natural isomorphisms that define the $G$-action by $(\mu_g)_{X,Y}:(X\otimes Y)^g\cong X^g\otimes Y^g$ and $(\gamma_{g,h})_X: X^{gh} \cong (X^g)^h$. Sometimes we will refer to these isomorphisms as the {\it tensorators} and {\it compositors} of the $G$-action. 

The explicit constraints satisfied by the $G$-action and $G$-braiding are as follows, where we take $X\in\cC_g$ and define $g^h=h^{-1}gh$ for notational convenience.
\begin{enumerate}
    \item From the constraint $((X\ot Y)\ot Z)^g\cong X^g\ot(Y^g\ot Z^g)$: $$\alpha_{X^g,Y^g,Z^g}\circ((\mu_g)_{X,Y}\ot \id_{Z^g})\circ(\mu_g)_{X\ot Y,Z}=(\id_{X^g}\ot (\mu_g)_{Y,Z})\circ(\mu_g)_{X,Y\ot Z}\circ\alpha_{X,Y,Z}.$$
    \item From the constraint $X^{khg} \cong ((X^k)^h)^g$: 
    $$((\gamma_{k,h})_X)^g\circ (\gamma_{kh,g})_{X}=(\gamma_{h,g})_{X^k}\circ (\gamma_{k,hg})_X$$
    \item From the constraint $(X\ot Y)^{gh}\cong (X^g)^h\ot (Y^g)^h$: 
    $$((\gamma_{g,h})_X\ot(\gamma_{g,h})_Y)\circ (\mu_{gh})_{X,Y}=(\mu_{h})_{X^g,Y^g}\circ((\mu_g)_{X,Y})^h\circ(\gamma_{g,h})_{X\ot Y}.$$
    \item From the constraint $(X\ot Y)^g\cong Y^g\ot X^{hg}$: $$(\id_{Y^g}\ot(\gamma^{-1}_{h,g})_X)\circ (\mu_g)_{Y,X^h}\circ (c_{X,Y})^g=(\id_{Y^g}\ot(\gamma^{-1}_{g,h^g})_{X})\circ c_{X^g,Y^g}\circ (\mu_g)_{X,Y}.$$
    \item From the constraint  $(X\ot Y)\ot Z\cong Z\ot (X^k\ot Y^k)$ where $Z\in \mathcal{C}_k$:
    $$\alpha_{Z,X^k,Y^k}\circ (c_{X,Z}\ot\id_{Y^k})\circ \alpha^{-1}_{X,Z,Y^k}\circ (\id_X\ot c_{Y,Z})\circ \alpha_{X,Y,Z}=
(\id_Z\ot (\mu_k)_{X,Y})\circ c_{X\ot Y,Z}.$$
    \item From the constraint $(X\ot Y)\ot Z\cong Y\ot (Z\ot (X^h)^k)$ where $Z\in \cC_k$ and $Y\in\cC_h$:
    $$\alpha_{Y,Z,(X^h)^k} \circ (\id_{Y\ot Z}\ot(\gamma_{h,k})_X)\circ c_{X,Y\ot Z}\circ\alpha_{X,Y,Z}=(\id_Y\ot c_{X^h,Z})\circ \alpha_{Y,X^h,Z}\circ (c_{X,Y}\ot \id_Z).$$
\end{enumerate}

Axioms (5) and (6) generalize the usual hexagon axioms of a braided tensor category and are often called either the {\it $G$-crossed hexagon} or {\it septagon axioms} \cite{BBCW}.

\begin{example}\label{ex: braided as G-crossed}
Any $G$-graded braided fusion category $\ca$ is trivially a $G$-crossed braided fusion category by taking the action of $G$ to be trivial.
\end{example}

\begin{remark}
\label{rmk:trivialization}
Conversely, a $G$-crossed braided fusion category $\cC$ is braided when its $G$-action $T:\underline{G} \to \underline{\Aut_{\otimes}(\cC)}$ admits a trivialization, {\it i.e.},  a monoidal natural isomorphism from the $G$-action functor $T$ to the trivial action $\Id_{\cC}$ \cite{Galindo2022}. We give a precise definition later in Section \ref{sec:recoveringbraidedzesting}.
\end{remark}

\begin{example}
\label{ex:vecgomega}
  Let $G$ be a finite group, $\omega\in Z^3(G,\ku^\times)$ a $3$-cocycle. The fusion category $\Vect_G^\omega$ of $G$-graded $\Bbbk$-vector  spaces with associativity constraint defined via $\omega$ can be given the structure of a $G$-crossed braided fusion category as follows (see \cite{Naidu11} for the left action version).  The category $\Vect_G^\omega$ has the (universal) grading by $G$ so that each component has one simple object which we identify with $g$.  On objects the $G$-action is given by $(g)^h:=h^{-1}gh$, with 
  \begin{align*}
(\mu_g)_{h,k}=\frac{\omega(h,k,g)\omega(g,h^g,k^g)}{\omega(h,g,k^g)}\id_{{g^{-1}hkg}} &\quad \text{ and }\quad
 (\gamma_{g,h})_{k}=\frac{\omega(g,k^g,h)}{\omega(g,h,k^{gh})\omega(k,g,h)}\id_{(gh)^{-1}kgh}.
 \end{align*}
 Taking $c_{g,h}=\id_{gh}:{g}\otimes h\cong h\otimes {h^{-1}gh}$ this yields a $G$-braiding.  
\begin{remark}
We will see in Section \ref{sec:Gcrossedzesting} that the  associative zesting construction in \cite{DGPRZ} is a major source of examples of $G$-crossed braided fusion categories.
\end{remark}
\end{example}

The appropriate notion of equivalence for $G$-crossed braided fusion categories is as follows:
\begin{definition}[Adapted from Section 3.1 of \cite{Galindo2017}] \label{def:G-equiv}
Let $\cC$ and $\cD$ be $G$-crossed braided categories, with $G$-action, compositors, and braiding $(T^\cC, \gamma^\cC,c^\cC)$ and $(T^\cD, \gamma^\cD,c^\cD)$ respectively. A triple $(F,J, \eta):\cC\rightarrow\cD$ is a $G$-crossed braided functor if $(F,J):\cC\rightarrow\cD$ is a monoidal functor and $\eta_g:T_g^{\cD}\circ F\rightarrow F\circ T_g^{\cC}$ is a family of monoidal natural isomorphisms for all $g \in G$ with $\eta_e=\Id_F$ such that
\begin{enumerate}
\item $F(\cC_g)\subset \cD_g$,
\item  $F((\gamma^{\cC}_{g,h})_X)\circ(\eta_{gh})_X=(\eta_h)_{X^g}\circ ((\eta_g)_{X})^h\circ (\gamma^{\cD}_{g,h})_{F(X)},$ and
\item $J_{Y,X^h}\circ F(c_{X,Y}^\cC)= (\id_{F(Y)} \ot (\eta_h)_X )\circ c_{F(X),F(Y)}^\cD\circ J_{X,Y}$
\end{enumerate}
for all $X \in \cC$, $Y \in \cC_h$. The triple $(F,J,\eta)$ is an {\it equivalence of $G$-crossed braided categories} if $F$ is an equivalence of categories.
\end{definition}
Recall from \cite[Definition 2.7.2]{EGNObook}, \cite{BN2013} the notion of a $G$-equivariantization.
\begin{definition}
\label{def:equivariantization}

Let $\cC$ be a fusion category with $G$-action $T$. The $G$-equivariantization $\cC^G$ is the fusion category with objects $(X,\{u_g\})$ where $X$ is an object of $\cC$ and $u_g: X^g \overset{\sim}{\rightarrow} X$ are isomorphisms for each $g \in G$ satisfying
\begin{equation}
  u_{gh} = u_h \circ (u_g)^h \circ(\gamma_{g,h})_X
\end{equation}
for all $h \in G$. Morphisms $f \in \Hom_{\cC^G}((X,\{u_g\}),(Y,\{v_g\}))$ are those $f \in \Hom_\cC(X,Y)$ satisfying $f \circ u_g = v_g \circ f^g$. The tensor product  is given by $(X,\{u_g\})\otimes (Y,\{v_h\})=(X\otimes Y,\{(u\cdot v)_g\})$, where $$(u\cdot v)_g:(X\otimes Y)^g\xrightarrow{\mu_g}X^g\otimes Y^g\xrightarrow{u_g\otimes v_g} X\otimes Y$$
\end{definition}

\begin{remark}\label{RMK: braiding of equivariantion of G-crossed}
When $\cC$ is $G$-crossed braided with $G$-braiding $c_{X_g,Y_h}: X_g \otimes Y_h \to Y_h \otimes X_g^h$ the equivariantization $\cC^G$ is braided with braiding
$c^G_{(X_g, \{u_k\}),(Y_h, \{v_k\})}$ induced by the isomorphism 

\[X_g\otimes Y_h \xrightarrow{c_{X_g,Y_h}} Y_h\otimes X_g^h  \xrightarrow{\id_{Y_h} \otimes u_h} Y_h\otimes X_g. \]
See \cite{MugerGalois} for details. 
\end{remark}

The following result is essentially contained in \cite{MugerGalois}, but we include a proof for completeness.
\begin{lemma}
\label{lem:equivariantization}
The $G$-equivariantizations $\cC^G,\cD^G$ of two equivalent $G$-crossed braided categories $\cC,\cD$ are equivalent as braided fusion categories.
\end{lemma}
\begin{proof}
Let $(F,J,\eta):\cC\to \cD$ be a $G$-crossed braided equivalence. Using the Coherence Theorems for $G$-crossed braided categories and tensor functors, we will assume that $\cC$, $\cD$, and $F$ are strict, that is the monoidal structure and $G$-action on each category is strict and $J$ is the identity, see \cite{Galindo2017}. Thus condition (2) in Definition \ref{def:G-equiv} gives the following 1-cocycle type equation
\[
(\eta_{gh})_X=(\eta_{h})_{X^g}\circ ((\eta_g)_X)^h.
\]
This equation allows us to define a canonical tensor functor $\widetilde{F}:\cC^G\to \cD^G$  as follows:

If $(X,\{u_g\})\in \cC^G$, then $\widetilde{F}((X,\{u_g\})):=(F(X),\{F(u_g)\circ (\eta_g)_X\})\in \cD^G$. In fact,
\begin{align*}
F(u_{gh})\circ (\eta_{gh})_X &=F(u_h)\circ F(u_g^h)\circ (\eta_{gh})_X\\
&=F(u_h)\circ F(u_g^h)\circ (\eta_{h})_{X^g}\circ ((\eta_g)_X)^h\\
&=F(u_h)\circ (\eta^h)_X \circ F(u_g)^h\circ ((\eta_g)_X)^h\\
&=\big ( F(u_h)\circ (\eta^h)_{X}\big )\circ \big ( F(u_g)\circ ((\eta_g)_X) \big )^h,
\end{align*}
where the third equality follows from the naturality of $\eta$. Again, using that $\eta$ is natural and monoidal, we can see that $\widetilde{F}$ defines a strict tensor functor. The functor $\widetilde{F}$ is an equivalence because $F$ is a $G$-equivariant monoidal equivalence. 

Finally, we will check that $\widetilde{F}$ is braided. Let $(X,\{u_g\}), (Y,\{v_h\})\in \cC^G$ and recall from Remark \ref{RMK: braiding of equivariantion of G-crossed} that the braiding on $\cC^G$ is induced by $c^G_{(X,u_g),(Y, v_h)}$ via

\[X_g\otimes Y_h \xrightarrow{c_{X_g,Y_h}} Y_h\otimes X_g^h  \xrightarrow{\id_{Y_h} \otimes u_h} Y_h\otimes X_g. \]

Condition (3) in Definition \ref{def:G-equiv} is $F(c_{X_g,Y_h})=(\id_{F(Y_h)}\otimes (\eta_h)_{X_g})\circ c_{F(X_g),F(Y_h)}$. Hence

\begin{align*}
    F\left( (\id_{Y_h}\otimes u_h)\circ c_{X_g,Y_h}\right)&=\left( \id_{F(Y_h)}\otimes F(u_h)\right)\circ F(c_{X_g,Y_h})\\
    &= \left( \id_{F(Y_h)}\otimes F(u_h)\right) \circ (\id_{F(Y_h)}\otimes (\eta_h)_{X_g})\circ c_{F(X_g),F(Y_h)}\\
    &=\left( \id_{F(Y_h)}\otimes (F(u_h)\circ (\eta_h)_{X_g})\right) \circ c_{F(X_g),F(Y_h)},
\end{align*}
which shows that the functor $\widetilde{F}$ respects the braiding.
\end{proof}

\begin{example}
\label{ex:vecgomegabraidings}
Any two $G$-braidings on $\Vect_G^\omega$ with the same $G$-action (as in Example \ref{ex:vecgomega}) are equivalent. Indeed, inspecting the septagon equations we see that any two such $G$-braidings $c,c^\prime$ are related by a bicharacter, say $c_{g,h}=\beta(g,h)c^\prime_{g,h}$.
  Let $(F,J)=(\Id,\id)$ and $(\eta_g)_h=\beta(g,h)\cdot \id_{g^{-1}hg}$.   We see that Condition (1) of Definition \ref{def:G-equiv} is clearly satisfied and (3) is immediate from comparing the braidings.  For (2) we take $X=k\in G$ so the equation reduces to verifying $\beta(gh,k)=\beta(h,g^{-1}kg)\beta(g,k)=\beta(h,k)\beta(g,k)$ which follows from the fact that $\beta$ is a bicharacter.

In particular, for an abelian group $A$, any two braidings on $\Vect_A$ yield equivalent $A$-crossed braided categories, although they can be distinct as (ordinary) braided categories.  This illustrates the stark difference between braided monoidal equivalence and crossed braided equivalence: a non-degenerately braided structure on $\Vect_A$ (for $|A|$ odd, say) and the standard Tannakian symmetric braiding yield equivalent $A$-crossed braided categories.
\end{example}

\subsection{Associative and braided zesting}
\label{subsec:braidedzesting}
Given an $A$-graded braided fusion category $\mathcal{B}$, one can modify its fusion rules by a 2-cocycle $\lambda$ valued in trivially-graded invertible objects and ask when the new fusion rules admit a monoidal structure. This is the idea of an {\it associative zesting}--one adjusts the associativity maps of $\cB$ by a $3$-cochain $\nu$ to obtain a new fusion category $\cB^{(\lambda,\nu)}$ when cohomological obstructions to doing so vanish. Then one can ask if the braiding in $\cB$ can be similarly adjusted to equip $\cB^{(\lambda,\nu)}$ with the structure of a braided fusion category. This is the idea of the {\it braided zesting} construction \cite{DGPRZ}. Here we use $\Inv(\mathcal{B})$ to mean the set of invertible objects rather than the group of isomorphism classes of invertible objects in $\mathcal{B}$. 

\begin{definition}
\label{def:braidedzesting}
Let $\mathcal{B}$ be an $A$-graded braided fusion category.\footnote{In \cite{DGPRZ} we considered a slightly more general setting with no braiding on the trivial component assumed.}

\begin{enumerate}[label=(\alph*)]
    \item an {\it associative zesting datum} $(\lambda,\nu)$ of $\cB$ \footnote{
In the definition of associative zesting given in \cite[Definition 3.2(2)]{DGPRZ} we denoted $\nu_{g_1,g_2,g_3}$ by $\lambda_{g_1,g_2,g_3}$ and $t_{g_1,g_2}$ by $t(g_1,g_2).$} consists of 
    \begin{itemize}
        \item a map $\lambda:A\times A\rightarrow \Inv(\mathcal{B}_e)$ which satisfies the 2-cocycle condition when $\Inv(\mathcal{B}_e)$ is interpreted as a group and \item isomorphisms $\nu_{g_1,g_2,g_3}:\lambda(g_1,g_2) \otimes \lambda(g_1g_2,g_3) \longrightarrow \lambda(g_2,g_3) \otimes \lambda(g_1, g_2g_3)$
        satisfying the {\it associative zesting condition}
        \begin{align}
        \label{eq:associativezestingcondition}
        \begin{split}
            (\nu_{g_2,g_3,g_4} \otimes \id_{\lambda(g_1,g_2g_3g_4)})\circ (\id_{\lambda(g_2,g_3)} \otimes \nu_{g_1,g_2g_3,g_4})\circ (\nu_{g_1,g_2,g_3} \otimes \id_{\lambda(g_1g_2g_3,g_4)}) \\ = (\id_{\lambda(g_3,g_4)} \otimes \nu_{g_1,g_2,g_3g_4}) \circ (c_{\lambda(g_1,g_2),\lambda(g_3,g_4)}^{-1} \otimes \id_{\lambda(g_1g_2,g_3g_4))}) \circ (\id_{\lambda(g_1,g_2)} \otimes \nu_{g_1g_2,g_3,g_4})
            \end{split}
        \end{align}
        for all $g_1,g_2,g_3,g_4 \in A$. {Furthermore we require that $\lambda(e,g_1)=\lambda(g_1,e)=\mathds{1}$ and $\nu_{g_1,e,g_2}=1$ for all $g_1,g_2 \in A$.}
    \end{itemize}
    \item A {\it braided zesting datum}  $(\lambda,\nu, t, j)$ of $\mathcal{B}$ is an associative zesting datum $(\lambda,\nu)$ together with
    \begin{itemize}
        \item  isomorphisms $t_{g_1,g_2}: \lambda(g_1,g_2) \longrightarrow \lambda(g_2,g_1)$, and
   
    \item a function $j:A \to \Aut_{\otimes}(\id_{\mathcal{B}})$
    satisfying the {\it braided zesting conditions} 
    \begin{equation}
    \begin{split}
        \nu_{g_2,g_3,g_1} \circ ( j_{g_1}(\lambda(g_2,g_3)) \otimes t_{g_1,g_2g_3)} )\circ \nu_{g_1,g_2,g_3} \\= (t_{g_1,g_2} \otimes \id_{\lambda(g_1g_2,g_3)})\circ \nu_{g_2,g_1,g_3} \circ (t_{g_1,g_3} \otimes \id_{\lambda(g_1g_3,g_2)})
    \end{split}
    \end{equation}
    and 
    \begin{equation}
    \label{eq:braidedzestingcondition2}
    \begin{split}
        \omega(g_1,g_2)(g_3)(\nu_{g_3,g_1,g_2}^{-1} \circ (\id_{\lambda(g_1,g_2)} \otimes t_{g_1g_2,g_3}) \circ \nu_{g_1,g_2,g_3}^{-1}) \\= (t_{g_1,g_3} \otimes \id_{\lambda(g_1g_3,g_2)})\circ \nu_{g_1,g_3,g_2}^{-1} \circ (t_{g_2,g_3} \otimes \id_{\lambda(g_1,g_2g_3)}) \end{split}
    \end{equation}
    where $\omega(g_1,g_2)= \chi_{\lambda(g_1,g_2)}\circ j_{g_1g_2}\circ j_{g_1}^{-1} \circ j_{g_2}^{-1}$ for all $g_1,g_2,g_3 \in A$. Here $\chi_{\lambda(g_1,g_2)} \in \Aut(\Id_{\mathcal{B}})$ is the natural isomorphism of the identity functor defined by $\chi_{\lambda(g_1,g_2)}(X) \otimes \id_{\lambda(g_1,g_2)}=  c_{\lambda(g_1,g_2),X} \circ c_{X,\lambda(g_1,g_2)}$.
       \end{itemize}
\end{enumerate}
\end{definition}

Associative zesting data $(\lambda_1,\nu_1)$ and $(\lambda_2,\nu_2)$ are considered inequivalent if $\lambda_1$ and $\lambda_2$ correspond to non-cohomologous $2$-cocycles or if $\nu_1$ and $\nu_2$ correspond to non-cohomologous $3$-cochains.  Certainly equivalent zesting data yield zestings that are monoidally equivalent.
On the other hand, even if $\lambda_1$ and $\lambda_2$ are non-cohomologous they can give rise to equivalent categories. The same is true for braided zestings: cohomologically distinct data can yield equivalent braided categories, see Example \ref{ex:su33}.

\begin{remark}\label{rem:reverse zesting} In \cite{DGPRZ} we picked the convention that the associative zesting condition (see Eqn. \ref{pic:associativezestingcondition}) and resulting zested associator involves the {\it inverse} braiding $c^{-1}_{Z_{g_3},\lambda(g_1,g_2)}$ see Eqn. (\ref{eq:zestedassociator}). However, we could also use  $c_{\lambda(g_1,g_2),Z_{g_3}}$ in the definition--the proofs go through {\it mutatis mutandis}, calling it the {\it reverse} associative zesting.  Indeed, if $(\lambda,\nu)$ is an associative zesting data then $(\lambda,\nu^{-1})$ is a reverse associative zesting.
\end{remark}

\begin{theorem}[Proposition 4.4 in \cite{DGPRZ}]
\label{thm:braidedzesting}
Let $(\lambda,\nu, t,j)$ be a braided zesting datum of a $\mathcal{B}$ as above. Then the category $\mathcal{B}^{(\lambda,\nu,t,j)}$ is a braided fusion category equal to $\mathcal{B}$ as an abelian category but with the monoidal structure in Equation \ref{eq:zestedassociator} 
and the braiding in Equation \ref{eq:zestedbraiding}. 
\end{theorem}
\qed

We call the category $\mathcal{B}^{(\lambda, \nu, t,j)}$ a {\it braided zesting} of $\mathcal{B}$. We will restrict our attention to the case when $j$ is trivial -- as is the case when $A$ is the universal grading of $\mathcal{B}$ -- and simply write $(\lambda, \nu, t)$ and $\mathcal{B}^{(\lambda,\nu,t)}$. 

\begin{example}[Adapted from Example 6.3.1 in \cite{DGPRZ}]
\label{ex:su33} For the rank 10 modular category $SU(3)_3$ there are three pairwise non-isomorphic invertible objects in  $[SU(3)_3]_e$ 
 with universal grading group $\Z/3\Z$.  
The (braided) zestings of $SU(3)_3$ with respect to the universal $\Z/3\Z$ grading group were computed in \cite[Example 6.3.1]{DGPRZ}.  Each of the 3 cohomology classes $\lambda\in H^2(\Z/3\Z,\Z/3\Z)$ admits $3$ choices of $3$-cochain $\nu$ for a total of $9$ cohomologously distinct associative zestings. 
However, not all $9$ of these associative zestings admit braidings: for each such $\lambda$ there is a unique (up to coboundaries) choice of $\nu\in C^3(\Z/3\Z,\Bbbk^\times)$ so that $(\lambda,\nu)$ can be promoted to a braided zesting datum $(\lambda,\nu, t)$.  In each such case, there are in fact, 3  promotions so that we get 9 braided zestings, all of which are non-degenerate.  It should be noted that they are not inequivalent: we believe that there are only 3 inequivalent braided fusion categories among these 9.
\end{example}

\subsubsection{Diagrammatic summary of the braided zesting construction}

\begin{remark}
\label{sec:methods}
Without loss of generality we may assume that our initial $G$-crossed braided category $\cC$ is {\it strict} \cite[Theorem 1.1]{Galindo2017}: that is, $\cC$ is strict as a monoidal category and the natural transformations $\mu_g$ and $\gamma_{g,h}$ that define the $G$-action are identities. Alternatively, one may assume that $\cC$ is {\it skeletal}, in which case there is one object per isomorphism class of simple objects and bases of all $\Hom$ spaces. Then $\cC$ is described by the data $\{N^{a_gb_h}_{c_{gh}},[F^{a_gb_hc_k}_{d_{ghk}}]_{(n_{hk},\gamma,\delta); (m_{gh},\alpha,\beta)}, [R^{a_gb_h}_{c_{gh}}]_{\mu\nu}, [\mu_k(a_g,b_h;c_{gh})]_{\mu\nu}, \gamma_{h,k}(a_g)\}$ consisting of the skeletal data for $\cC$ as a fusion category together with the matrix entries of the $G$-crossed braiding, tensorators, and compositors. These symbols satisfy skeletal analogs of Equations (1)-(6) above and obey a different graphical calculus than the one used here, in which morphisms are represented as linear combinations of trivalent ribbon graphs, see \cite[Equation (277, 279, 273, 282, 286, 287)]{BBCW} respectively, in which $\mu$ is denoted by $U$ and $\gamma$ by $\eta$.  

Here we will typically assume $\cC$ is strict when building general theory. Some of the examples will feature non-strict $\cC$ (as in the earlier Example \ref{ex:vecgomega}) but they are all simple enough to analyze using the categorical notation of Definition \ref{def:gcrossedbraided}, so we will get by without writing down a more careful definition of a skeletal $G$-crossed braided fusion category.  
\end{remark}

There are two main features that make zesting a useful paradigm for studying braided fusion categories. First, Equations \ref{eq:associativezestingcondition} - \ref{eq:braidedzestingcondition2} are equalities of isomorphisms of simple objects, and can thus be interpreted as scalar equations. Second, the zesting construction can be derived entirely diagrammatically. In fact, the graphical calculus for $G$-crossed braided tensor categories is our main tool in Section \ref{sec:Gcrossedzesting}. It extends the standard graphical calculus for tensor categories (see Section 2 of \cite{DGPRZ} for a review consistent with the conventions used here) with additional pictures and local relations to accommodate the $G$-action and $G$-crossed braiding. Diagrams are to be read from the top down and the following conventions are taken when drawing the $G$-action and $G$-crossed braiding, respectively.

$$
\begin{tikzpicture}[scale=.75, baseline=20]
\onebox{Y}{X}{f}
\draw (.4,.9) node[right] {$\cdot \,g$};
\end{tikzpicture}
:=
\begin{tikzpicture}[scale=.75, baseline=20]
\onebox{Y^g}{X^g}{f^g}
\end{tikzpicture}
\qquad \qquad c_{X_g,Y_h}= \begin{tikzpicture}[scale=.75, line width=1, baseline = 20]
            \draw (0,0) node[below] {$\phantom{^h}Y_h\phantom{^h}$} to [out=90, in=-90] (1,2) node[above] {$\phantom{_g}Y_h\phantom{_g}$};
             \draw[white, line width=10] (1,0) to [out=90, in=-90] (0,2);
             \draw (1,0) node[below] {$X_g^h$} to [out=90, in=-90] (0,2) node[above] {$X_g$};
        \end{tikzpicture}$$

\begin{remark}
\label{rmk:technical}
Let $\cC$ be a fusion category. The maximal pointed fusion subcategory $\cC_{\operatorname{pt}}$ is integral, so it follows from \cite[Propositions 8.23 and 8.24]{etingof2005fusion} that $\cC_{\operatorname{pt}}$ admits a unique spherical structure such that the quantum dimension of every simple object is one. Hence from now on we will assume without loss of generality that $\dim(X)=1$ for all $X \in \Inv(\cC)$. In particular if $(\lambda,\nu)$ is an associative zesting datum we will assume that $\dim(\lambda(g_1,g_2))=1$ for all $g_1,g_2 \in G$.
\end{remark}

Moreover, we may assume the spherical structure on $\cC_{pt}$ is strict, so that $^*X=X^*$ (and $X=X^{**})$. As a result the evaluation/coevaluation morphisms of $\cC$ restricted to the image of $\lambda$ in $\cC_{pt}$ can be drawn as follows \cite[Section 2.10]{EGNObook}:
\begin{align*} 
\text{eval}_{\lambda(g_1,g_2)}= &\hspace{7.5pt} \begin{tikzpicture}[scale=1.5,line width=1,baseline=-15]
\draw[looseness=2] (0,0) node[above] {\small $\lambda(g_1,g_2)^*$} to [out=-90, in=-90] (1,0) node[above] {\small $\phantom{\lambda^{*}}\lambda(g_1,g_2)\phantom{\lambda^{*}}$};
\end{tikzpicture} \quad & \text{coeval}_{\lambda(g_1,g_2)} =& \begin{tikzpicture}[scale=1.5,line width=1,baseline=0]
\draw[looseness=2] (0,0) node[below] {\small $\phantom{\lambda^{*}}\lambda(g_1,g_2)\phantom{\lambda^{*}}$} to [out=90, in=90] (1,0) node[below] {\small $\lambda(g_1,g_2)^{*}$};
\end{tikzpicture}  \\
\text{eval}'_{\lambda(g_1,g_2)}= &\begin{tikzpicture}[scale=1.5,line width=1,baseline=-15]
\draw[looseness=2] (0,0) node[above] {\small $\phantom{\lambda^{*}}\lambda(g_1,g_2)\phantom{\lambda^{*}}$} to [out=-90, in=-90] (1,0) node[above] {\small $\lambda(g_1,g_2)^*$};
\end{tikzpicture}  \quad &  \text{coeval}'_{\lambda(g_1,g_2)} =&\hspace{7.5pt}\begin{tikzpicture}[scale=1.5,line width=1,baseline=0]
\draw[looseness=2] (0,0) node[below] {\small $\lambda(g_1,g_2)^{*}$} to [out=90, in=90] (1,0) node[below] {\small $\phantom{\lambda^{*}}\lambda(g_1,g_2)\phantom{\lambda^{*}}$};
\end{tikzpicture} \end{align*}

See the following Table \ref{table:braidedzesting} for diagrammatic overview of the braided zesting construction given in Definition \ref{def:braidedzesting} in the case where $j$ is trivial. In the last row of the table the labels $\lambda(g,h)$, $\nu_{g,h,k}$, and $t_{g,h}$ are abbreviated by suppressing $\lambda$, $\nu$, and $t$ and keeping only their indices in $G$.

\begin{table}[ht]
\label{table:braidedzesting}
\begin{minipage}{.5\textwidth}
\begin{center}
\textbf{Associative Zesting}
\end{center}
\end{minipage}%
\begin{minipage}{.5\textwidth}
\begin{center}
\textbf{Braided Zesting}
\end{center}
\end{minipage}
\hrule
\hrule
\begin{minipage}{.5\textwidth}
\begin{center}
\scalebox{.75}{\begin{tikzpicture}[line width=1,scale=4/3,baseline=20*4/3]
  \foreach \x in {0,1}{
  \draw(\x,0)--(\x,2);
  }
  \draw[fill=white] (-.25,.75) rectangle node {$\nu_{g_1,g_2,g_3}$} (1.25,1.25) ;
   \draw (-.125,2) node[above] {$\lambda(g_1,g_2)$};
      \draw (1.125,2) node[above] {$\lambda(g_1g_2,g_3)$};    
  \draw (-.125,0) node[below] { $
 \lambda(g_2,g_3)$};
   \draw (1.125,0) node[below] {$\lambda(g_1,g_2g_3)$};
   \end{tikzpicture}}
   \end{center}
   \begin{center}
   Associative zesting datum $(\lambda,\nu)$
   \end{center}
\end{minipage}%
\begin{minipage}{.5\textwidth}
\begin{center}
\scalebox{.75}{\begin{tikzpicture}[line width=1,scale=4/3, baseline=20*4/3]
\draw (0,0) node[below] {$\lambda(g_2,g_1)$} -- (0,2) node[above] {$\lambda(g_1,g_2)$};
\draw[fill=white] (0,1) circle (.6) node {\small $t_{g_1,g_2}$};
\end{tikzpicture}}
\end{center}
\begin{center}
Braided zesting datum $(\lambda, \nu, t)$
\end{center}
\end{minipage}
\hrule
\begin{minipage}{.5\textwidth}
\begin{equation}
\label{eq:zestedassociator}
\scalebox{.75}{
\assocfull{g_1}{g_2}{g_3}{\nu}{.65}
}\end{equation}
\begin{center}
Zested associator (when $\mathcal{B}$ is strict)
\end{center}
\end{minipage}%
\begin{minipage}{.5\textwidth}
\begin{equation}
\label{eq:zestedbraiding}
\scalebox{.75}{\begin{tikzpicture}[line width=1,scale = 11/8*.65, baseline=75*.65,looseness=1.5]
\draw (0,0) node[below] {$Y_{g_2}$} \br (2,4) node[above] {$Y_{g_2}$}; 
\draw[white, line width=10] (2,0) \br (0,4);
\draw (2,0) node[below] {$X_{g_1}$} \br (0,4) node[above] {$X_{g_1}$};
\draw (4,0) node[below] {$\lambda(g_2,g_1)$} -- (4,4) node[above] {$\lambda(g_1,g_2)$};
\draw[fill=white] (4,2) circle (.9) node {$t_{g_1,g_2}$};
\end{tikzpicture}}
\end{equation}
\begin{center}
Zested braiding
\end{center}
\end{minipage}
\hrule
\begin{minipage}{.5\textwidth}
\begin{equation}
\label{pic:associativezestingcondition}
   \scalebox{.75}{\begin{tikzpicture}[line width=1,scale=1.5,baseline=3cm]
  \foreach \x in {0,1,2}{
  \draw(\x,0)--(\x,4);}
  \draw[fill=white] (-.25,.75) rectangle node {$g_2,g_3,g_4$} (1.25,1.25) ;
  \draw[fill=white] (.75,1.75) rectangle node {$g_1,g_2g_3,g_4$} (2.25,2.25);
  \draw[fill=white] (-.25,2.75) rectangle node {$g_1,g_2,g_3$} (1.25,3.25) ;
  \draw (0,0) node[below] {$(g_3,g_4)$};
   \draw (1,0) node[below] {$(g_2,g_3g_4)$};
    \draw (2.25,0) node[below] {$(g_1,g_2g_3g_4)$};
    \draw (0,4) node[above] {$(g_1,g_2)$};
      \draw (1,4) node[above] {$(g_1g_2,g_3)$};    
      \draw (2.25,4) node[above] {$(g_1g_2g_3,g_4)$};
  \end{tikzpicture} =
  \begin{tikzpicture}[line width=1,scale=1.5, baseline=3cm]
   \foreach \x in {0,1,2}{
  \draw(\x,0)--(\x,1);
  \draw(\x,2.5)--(\x,4);
  }
   \begin{scope}[xshift=-1cm,yshift=2.5cm]
  \braid[number of strands=3] a_1^{-1};
  \end{scope}
   \draw[fill=white] (.75,.75) rectangle node {$g_1,g_2,g_3g_4$} (2.25,1.25);
   \draw[fill=white] (.75,2.75) rectangle node {$g_1g_2,g_3,g_4$} (2.25,3.25);
  \draw (0,0) node[below] {$(g_3,g_4)$};
   \draw (1,0) node[below] {$(g_2,g_3g_4)$};
    \draw (2.25,0) node[below] {$(g_1,g_2g_3g_4)$};
    \draw (0,4) node[above] {$(g_1,g_2)$};
      \draw (1,4) node[above] {$(g_1g_2,g_3)$};    
      \draw (2.25,4) node[above] {$(g_1g_2g_3,g_4)$};
  \end{tikzpicture}}
   \end{equation}
  \begin{center}
Associative zesting condition
\end{center}
\end{minipage}%
\begin{minipage}{.5\textwidth}
\begin{equation}
\scalebox{.75}{
\begin{tikzpicture}[line width=1, baseline=3cm]
\draw (0,0)--(0,4);
\draw (1,0)--(1,4);
\twoboxnotop{}{}{g_2,g_3,g_1}
\draw (-.25,0) node[below] {$(g_3,g_1)$};
\draw (1.25,0) node[below] {$(g_3g_1,g_2)$};
\begin{scope}[yshift=2.67cm]
\twobox{}{}{}{}{g_1,g_2,g_3}
\draw (-.25,2) node[above] {$(g_1,g_2)$};
\draw (1.25,2) node[above] {$(g_1g_2,g_3)$};
\end{scope}
\draw (2.5,2.33) node[right] {$=$};
\begin{scope}[yshift=-6.67cm]
\draw (-.25,2.33) node[left] {$\textstyle \omega(g_1,g_2;g_3)\cdot$}; 
\draw (0,2)--(0,4);
\draw (1,0)--(1,4.67);
\draw (-.25,0) node[below] {$(g_3,g_1)$};
\draw (1.25,0) node[below] {$(g_3g_1,g_2)$};
\draw (-.25,4.67) node[above] {$(g_2,g_3)$};
\draw (1.25,4.67) node[above] {$(g_1,g_2g_3)$};
\twobox{}{}{}{}{g_3,g_1,g_2}
\draw (1.675,1.25) node {$^{-1}$};
\begin{scope}[yshift=2.67cm]
\twobox{}{}{}{}{g_1,g_2,g_3}
\draw (1.675,1.25) node {$^{-1}$};
\end{scope}
\draw (2.5,2.33) node[right] {$=$};
\end{scope}
\draw[fill=white] (1,2.33) node {$\scriptstyle g_1,g_2g_3$} circle (.5);
\draw[fill=white] (1,-4.33) node {$\scriptstyle g_1g_2,g_3$} circle (.5);
\end{tikzpicture}\begin{tikzpicture}[line width=1,baseline=3cm]
\draw (1,0) node[below] {}--(1,4.67) node[above] {};
\draw (1.25,0) node[below] {$(g_3g_1,g_2)$};
\draw (1.25,4.67) node[above] {$(g_1g_2,g_3)$};
\draw (-.25,4.67) node[above] {$(g_1,g_2)$};
\draw (-.25,0) node[below] {$(g_3,g_1)$};
 \begin{scope}[yshift=1.33cm]
\twobox{}{}{}{}{g_2,g_1,g_3}
 \end{scope}
 \onecirc{}{}{\scriptstyle g_1,g_3}
  \begin{scope}[yshift=2.67cm]
  \onecirc{}{}{\scriptstyle g_1,g_2}
 \end{scope}
 \begin{scope}[yshift=-6.67cm]
   \draw (1,0)--(1,4.67);
\draw (-.25,0) node[below] {$(g_3,g_1)$};
\draw (1.25,0) node[below] {$(g_3g_1,g_2)$};
\draw (-.25,4.67) node[above] {$(g_2,g_3)$};
\draw (1.25,4.67) node[above] {$(g_1,g_2g_3)$};
\begin{scope}[yshift=1.33cm]
\twobox{}{}{}{}{g_1,g_3,g_2}
\draw (1.675,1.25) node[] {$^{-1}$};
\end{scope}
\onecirc{}{}{\scriptstyle g_1,g_3}
\begin{scope}[yshift=2.67cm]
\onecirc{}{}{\scriptstyle g_2,g_3}
\end{scope}
\end{scope}
\end{tikzpicture}}
\end{equation}
\begin{center}
Braided zesting conditions
\end{center}
\end{minipage}
\caption{Diagrammatic overview of associative and braided zesting data. Per Remark \ref{rem:reverse zesting} the crossings are opposite those used in Section \ref{sec:Gcrossedzesting}.}
\end{table}

\newpage
\section{\texorpdfstring{$G$}{G}-crossed zesting}
\label{sec:Gcrossedzesting}
The construction of a $G$-crossed braided fusion category from zesting unfolds essentially the same as in the braided case. First one modifies the fusion rules and monoidal structure (Section \ref{sec:associativegcrossedzesting}), then the ($G$-crossed action and)  braiding  (Section \ref{sec:zestedgactionbraiding}). Additional structures (pivotality, traces, etc.) and properties (rigidity, sphericality, modularity) are considered in Section \ref{sec:zestedstructures} and \ref {sec:zestedmodulardata}.
\subsection{Associative \texorpdfstring{$G$}{G}-crossed zesting}
\label{sec:associativegcrossedzesting}
We begin with a definition generalizing that of an associative zesting datum:
\begin{definition}
\label{def:gcrossedassociativezesting}
Let $\cC$ be a $G$-crossed braided category. An {\it associative $G$-crossed zesting datum} is a pair $(\lambda,\nu)$ where $\lambda:G \times G \to \Inv(\cC_e)$ is a normalized 2-cocycle with $G$-action and $\nu$ are isomorphisms
$$\nu_{g_1,g_2,g_3}: \lambda(g_1,g_2)^{g_3} \otimes \lambda(g_1g_2,g_3) \to \lambda(g_2,g_3) \otimes \lambda(g_1,g_2g_3)$$
satisfying $\nu_{g_1,e,g_2} = 1$ and
\begin{align}
\label{eq:associativeGzesting}
\begin{split}
  \left(\nu_{g_2,g_3,g_4} \otimes \id_{\lambda(g_1,g_2g_3g_4)} \right)\circ \left (\id_{\lambda(g_2,g_3)^{g_4}} \otimes \nu_{g_1,g_2g_2,g_4} \right) \circ \left ((\nu_{g_1,g_2,g_3})^{g_4} \otimes \id_{\lambda(g_1g_2g_3,g_4)} \right) \\
  =\left ( \id_{\lambda(g_3,g_4)} \otimes \nu_{g_1,g_2,g_3g_4} \right) \circ \left ( c_{\lambda(g_1,g_2)^{g_3g_4},\lambda(g_3,g_4)} \otimes \id_{\lambda(g_1g_2g_3,g_4)}\right) \circ \left ( \id_{\lambda(g_1,g_2)^{g_3g_4}} \otimes \nu_{g_1g_2,g_3,g_4}\right)
  \end{split}
\end{align}
for all $g_1,g_2,g_3,g_4 \in G$.
We represent the isomorphisms $\nu$ and the equations they satisfy diagrammatically as 
\begin{align}
    \label{fig:associativeGzesting}
\begin{tikzpicture}[baseline=0]
\gcocyc{g_1}{g_2}{g_3}{\nu}
\end{tikzpicture}
\end{align}
and
\begin{equation} 
  \label{gcrossedassoczesting}
  \begin{tikzpicture}[line width=1,xscale=1.75, yscale=1.5, baseline=2.5cm]
  \foreach \x in {0,1,2}{
  \draw(\x,0)--(\x,4);
  }
  \draw[fill=white] (-.25,.75) rectangle node {$\nu_{g_2,g_3,g_4}$} (1.25,1.25) ;
  \draw[fill=white] (.75,1.75) rectangle node {$\nu_{g_1,g_2g_3,g_4}$} (2.25,2.25);
  \draw[fill=white] (-.25,2.75) rectangle node {$\nu_{g_1,g_2,g_3}$} (1.25,3.25) ;
  \draw (1.25,3) node[right] {$\cdot g_4$};
  \draw (0,0) node[below] {$\lambda(g_3,g_4)$};
   \draw (1,0) node[below] {$\lambda(g_2,g_3g_4)$};
    \draw (2.25,0) node[below] {$\lambda(g_1,g_2g_3g_4)$};
    \draw (-.25,4) node[above] {$\lambda(g_1,g_2)^{g_3g_4}$};
      \draw (1,4) node[above] {$\lambda(g_1g_2,g_3)^{g_4}$};    
      \draw (2.25,4) node[above] {$\lambda(g_1g_2g_3,g_4)$};
  \end{tikzpicture} 
  =
  \begin{tikzpicture}[line width=1,xscale=1.75, yscale=1.5, baseline=2.5cm]
   \foreach \x in {0,1,2}{
  \draw(\x,0)--(\x,1);
  \draw(\x,2.5)--(\x,4);
  }
   \begin{scope}[xshift=-1cm,yshift=2.5cm]
  \braid[number of strands=3] a_1;
  \end{scope}
   \draw[fill=white] (.75,.75) rectangle node {$\nu_{g_1,g_2,g_3g_4}$} (2.25,1.25);
   \draw[fill=white] (.75,2.75) rectangle node {$\nu_{g_1g_2,g_3,g_4}$} (2.25,3.25);
  \draw (0,0) node[below] {$\lambda(g_3,g_4)$};
   \draw (1,0) node[below] {$\lambda(g_2,g_3g_4)$};
    \draw (2.25,0) node[below] {$\lambda(g_1,g_2g_3g_4)$};
    \draw (-.25,4) node[above] {$\lambda(g_1,g_2)^{g_3g_4}$};
      \draw (1,4) node[above] {$\lambda(g_1g_2,g_3)^{g_4}$};    
      \draw (2.25,4) node[above] {$\lambda(g_1g_2g_3,g_4)$};
  \end{tikzpicture}.
  \end{equation}
  \end{definition}
  \begin{remark}
  Observe that the crossing in the $G$-crossed associative zesting condition is the reverse of the one that appears in the associative zesting condition of \cite[Figure 2]{DGPRZ} due to the convention we have taken for defining the $G$-crossed braiding.
  \end{remark}

\begin{proposition}
\label{prop:associativegzesting}
Let $(\lambda,\nu)$ be a $G$-crossed associative zesting datum of a $G$-crossed braided fusion category $\cC$. Then $\cC$ (as an abelian category) also admits a monoidal structure with the tensor product $X_{g_1} \overset{\lambda}{\otimes} Y_{g_2} = X_{g_1} \otimes Y_{g_2} \otimes \lambda(g_1,g_2)$. 
\end{proposition}

\begin{proof} Without loss of generality assume that $\cC$ is strict. Define
$\alpha^{(\lambda,\nu)}_{X_{g_1},Y_{g_2},Z_{g_3}}:(X_{g_1} \overset{\lambda}{\otimes} Y_{g_2}) \overset{\lambda}{\otimes} Z_{g_3} \rightarrow X_{g_1} \overset{\lambda}{\otimes}(Y_{g_2} \overset{\lambda}{\otimes}Z_{g_3})$
by the diagram on the right-hand side below. 
   \begin{align}
   \label{eq:gcrossedzestedassociator}
 \assocrev{g_1}{g_2}{g_3}{.65} &:=&
\assocrevfull{g_1}{g_2}{g_3}{\nu}{.65}.
   \end{align}

As in Table \ref{table:braidedzesting} we will use the shorthand notation indicated above to save space. The normalization conditions satisfied by $\lambda$ and $\nu$ ensure that the left and right unit axioms hold for $\mathds{1} \in \cC$ with respect to $\alpha^{(\lambda,\nu)}$, while the pentagon axiom with respect to the $\alpha^{(\lambda,\nu)}$ takes the form of the equation
\begin{equation}
\begin{tikzpicture}[line width=1, looseness=.75, xscale=1.1,yscale=.5,baseline=60]
\draw (.75,0) node[below] {\small $\parcen{g_1}$}--(.75,8)node[above] {\small $\parcen{g_1}$};
\draw (1.25,0)node[below] {\small $\parcen{g_2}$}--(1.25,8) node[above] {\small $\parcen{g_2}$};
\draw (6,0) --(6,8);
\draw (6.5,0) node[below] {\small $(g_1,g_2g_3g_4)$};
\draw (6.25,8) node[above] {\small $(g_1g_2g_3,g_4)$};
\draw (5,0) --(5,3);
\draw[looseness=1.25] (3,0)  \br (5,8);
\draw (3,0)node[below] {\small $\parcen{g_4}$};
\draw(5,8) node[above] {\small $\parcen{g_4}$};
\draw[white, line width=10,looseness=1.25] (5,3) \br (4,6);
\draw[looseness=1.25] (5,3) \br (4,6)--(4,8) node[above] {\small $ (g_1g_2,g_3)$};
\draw (2,0)[looseness=1.25] node[below] {\small $\parcen{g_3}$} --(2,6) \br (3,8) node[above] {\small $\parcen{g_3}$};
\draw[white, line width=10] (3,6) \br (2,8);
\draw[white, line width=10,looseness=1.25] (4,2) \br (3,5);
\draw[looseness=1.25] (4,1)--(4,2) \br (3,5) --(3,6) \br (2,8) node[above] {\small $(g_1,g_2)$};
\begin{scope}[xshift=4cm]
\twoboxnotop{}{}{\scriptstyle g_2,g_3,g_4}
\draw (1,0) node[below] {\small $(g_2,g_3g_4)$};
\end{scope}
\draw (3.75,0) node[below] {\small $(g_3,g_4)$};
\begin{scope}[xshift=5cm, yshift=2cm]
\twoboxnotop{}{}{\scriptstyle g_1,g_2g_3,g_4}
\end{scope}
\begin{scope}[xshift=3cm,yshift=5cm]
\twoboxnostrands{\scriptstyle g_1,g_2,g_3}
\end{scope}
\draw (7,4) node {$=$};
\end{tikzpicture}
    \begin{tikzpicture}[line width=1, looseness=.75, xscale=1.21,yscale=.5,baseline=60]
      \draw (.75,0) node[below] {\small $\parcen{g_1}$}--(.75,8)node[above] {\small $\parcen{g_1}$};
    \draw (1.25,0)node[below] {\small $\parcen{g_2}$}--(1.25,8) node[above] {\small $\parcen{g_2}$};
    \draw (6,0) --(6,8);
    \draw (6.5,0) node[below] {\small $(g_1,g_2g_3g_4)$};
    \draw (6.25,8) node[above] {\small $(g_1g_2g_3,g_4)$};
    \draw[looseness=.75] (3,0) node[below] {\small $\parcen{g_4}$} \br (5,8) node[above] {\small $\parcen{g_4}$};
    \draw[white, line width=10,looseness=1.25] (5,4.675) \br (4,8);
    \draw[looseness=1.25] (5,4.675) \br (4,8) node[above] {\small $(g_1g_2,g_3)$};
  \draw[looseness=1.25] (2,0) node[below] {\small $\parcen{g_3}$}\br (3,8) node[above] {\small $\parcen{g_3}$};
  \draw[looseness=1.25] (4,0)  \br (5,4.675);
  \draw (3.75,0) node[below] {\small $(g_3,g_4)$};
  \draw[white, line width=10, looseness=1.25] (5,1.33) \br (2,6)--(2,8); 
    \draw[looseness=1.25] (5,1.33) \br (2,6)--(2,8) node[above] {\small $(g_1,g_2)$};
       \begin{scope}[xshift=5cm]
     \twoboxnotop{}{}{\scriptstyle g_1,g_2,g_3g_4}
     \draw (0,0) node[below] {\small $(g_2,g_3g_4)$};
    \begin{scope}[yshift=4cm]
     \twoboxnostrands{\scriptstyle g_1g_2,g_3,g_4}
 \end{scope}
  \end{scope}
    \end{tikzpicture}.
    \end{equation} After sliding the strand labeled $4$ above the box $\nu_{g_1,g_2,g_3}$ in the picture on the left, equality follows immediately from Equation \ref{fig:associativeGzesting}. Other than minding the $G$-action the argument is identical to the proof of \cite[Proposition 3.4]{DGPRZ}.
\end{proof}
We denote by $\cC^{(\lambda,\nu)}$ the monoidal category obtained from an associative $G$-crossed zesting datum $(\lambda,\nu)$. We will see later that $\cC^{(\lambda,\nu)}$ has a canonical $G$-crossed braided structure. As in Remark \ref{rmk:technical} we may assume that $\dim(\lambda(g_1,g_2))=1$ for all $g_1,g_2 \in G$.
\begin{remark}\label{remark H4 obstruction}

Given $\lambda \in Z^2(G,\Inv(\cC_e))$ and a 3-cochain $\nu \in C^3(G,\mathds{k}^{\times})$, the obstruction $O_{\lambda}$ to $\cC$ forming a monoidal category with associator equivalent to that in Equation \ref{eq:zestedassociator} is measured by the following diagram. Note that when the $G$-action is trivial one recovers the obstruction to associative zesting from \cite[Figure 5]{DGPRZ}\footnote{Actually there is an error in the crossing used in Figure 5 of the published version of \cite{DGPRZ}, it should be the reverse.}
\begin{equation}
  \begin{tikzpicture}[line width=1,xscale=1.33,yscale=.6]
  \draw (-.5,5) node[left] {$O_{\lambda}(\nu):=$};
    \draw (1,1.25) to [out=90, in=-90] (0,3.25)--(0,4);
    \draw[white, line width=10] (0,0) -- (0,1) to [out=90,in=-90] (1,2.75);
    \draw(0,0) -- (0,1) to [out=90,in=-90] (1,2.75);
    \draw (-.5,0) node[below] {\small $(g_1,g_2)^{g_3g_4}$};
    \draw (-.5,10) node[above] {\small $(g_1,g_2)^{g_3g_4}$};
    \draw (0,6)--(0,8);
    \draw (2,1.25)--(2,10);
    \draw (2.5,10) node[above] {\small $(g_1g_2g_3,g_4)$};
    \begin{scope}[xshift=1cm]
    \twoboxnotop{}{}{\scriptstyle g_1g_2,g_3,g_4}
    \draw (0,0) node[below] {\small $(g_1g_2,g_3)^{g_4}$};
    \draw (1.5,0) node[below] {\small $(g_1g_2g_3,g_4)$};
    \draw (1.5,1.25) node[] {\small $^{-1}$};
    \end{scope}
     \begin{scope}[xshift=1cm,yshift=2cm]
    \twoboxnobottom{}{}{\scriptstyle g_1,g_2,g_3g_4}
    \draw (1.5,1.25) node[] {\small $^{-1}$};
    \end{scope}
     \begin{scope}[yshift=4cm]
     \twobox{}{}{}{}{\scriptstyle g_2,g_3,g_4}
     \end{scope}
       \begin{scope}[xshift=1cm,yshift=6cm]
     \twobox{}{}{}{}{\scriptstyle g_1,g_2g_3,g_4}
     \end{scope}
       \begin{scope}[yshift=8cm]
    \draw (1.33,1) node[right] {\small $\cdot g_4$};
     \twobox{}{}{}{}{\scriptstyle g_1,g_2,g_3}
     \draw (1,2) node[above] {\small $(g_1g_2,g_3)^{g_4}$};
     \end{scope}
    \end{tikzpicture}
    \end{equation}
   
The map $O_\lambda$ enjoys identical properties to those in the non-$G$-crossed setting (\cite[Proposition 3.9]{DGPRZ}), which we now briefly recall.\footnote{Regrettably we denoted the analogue of $O_{\lambda}$ by $\nu$ in \cite{DGPRZ}.} The properties of $O_\lambda$ and its connection to $G$-crossed braided extension theory are established in \cite[Section 8.7]{ENO}, where it is explained that $O_\lambda$ is a 4-cocycle on $G$ whose cohomology class only depends on that of $\lambda$. The {\it Pontryagin-Whitehead quadratic function} is defines via
\begin{align}\label{defintion of PW}
    \text{PW}: H^2(G, \Inv(\cC_e)) &\to H^4(G, \mathds{k}^{\times})\\
     \beta &\mapsto O_\beta \notag
\end{align}
then measures whether a 2-cocycle $\lambda \in Z^2(G,\Inv(\cC_e))$ has a lifting (\cite[Definition 3.8]{DGPRZ}), {\it i.e.}, whether $\cC^{(\lambda,\nu)}$ is a fusion category. When $\text{PW}(\beta)=0$, the equivalence classes of these liftings form a torsor over $H^3(G,\mathds{k}^{\times})$.
\end{remark}

\subsection{Zested \texorpdfstring{$G$}{G}-crossed braiding and \texorpdfstring{$G$}{G}-action}
\label{sec:zestedgactionbraiding}
In the following sections we use $g,h,k$ etc. instead of $g_1,g_2,g_3$ etc. for ease of notation.
 
To see that there is a unique $G$-action and braiding on $\cC^{(\lambda,\nu)}$ up to equivalence it is helpful to have in mind the following equivalent way to look to $G$-crossed braided categories in terms of {\it central functors} $\cB \hookrightarrow \cZ(\cC)$ where $\cC_e=\cB$ is the trivial component $\cC$. We now recall that every $G$-crossed braided extension of $\mathcal{B}$ is equivalent to a \emph{central $G$-extension} of $\mathcal{B}$.

\subsubsection{From Central \texorpdfstring{$G$}{G}-extensions to \texorpdfstring{$G$}{G}-crossed braided tensor categories}\label{subsection: from central G-ext to G-crossed ext}

\begin{definition}
\label{def:centralextension}
Let $\cC$ be a (strict) fusion category and $\cB\subset \cC$  a fusion subcategory. A \emph{relative braiding} of $\cB$ in $\cC$ consists of a family of natural isomorphisms 

\[
\{c_{Z,X}:Z\otimes X\to X\otimes Z\}_{Z\in \cC, X\in \cB}
\]
such that the hexagon axioms
\begin{align}
c_{Z,X\otimes Y}=(\id_X\otimes c_{Z,Y} ) ( c_{Z,X}\otimes \id_Y), && c_{Z\otimes W,X}=(c_{Z,X}\otimes \id_W)(\id_Z\otimes c_{W,X})
\end{align}
are satisfied for all $X,Y \in \cB, Z,W\in \cC$.

An inclusion $\cB\subset \cC$ with a relative braiding is called a {\it central extension} of $\cB$. A \emph{$G$-graded central  extension} is a central extension with a faithful $G$-grading.
\end{definition}

\begin{remark}
In \cite[Definition 8.9]{DN} the authors define a central $G$-extension as a  $G$-graded extension $\cB\subset \cC$ with the structure of a central functor on the inclusion functor. A central functor on any inclusion $\cB\subset \cC$ is a braided tensor functor
$\iota :\cB\to \mathcal{Z}(\cC)$ such that $U\circ \iota =\id_{\cB}$ where $U:\mathcal{Z}(\cC)\to \cC$ is the forgetful functor.  This can be seen to be equivalent to our notion of a $G$-graded central extension so that the similar nomenclature should not cause any confusion.  Indeed, a central extension $\cB\subset \cC$ as above yields a braiding on $\cB$ and $A\mapsto (A,c_{-,A})$ defines a braided functor $i :\cB\to \mathcal{Z}(\cC)$ (Here we are defining the center via right-half-braidings.) such that $U\circ i=\id_\cB$.
 Conversely, an inclusion $\iota:\cB\to\mathcal{Z}(\cC)$ such that $U\circ\iota =\id_{\cB}$ defines a relative braiding and thus a central extension.
\end{remark}

It is clear that if $\cC$ is a $G$-crossed braided category with trivial component $\cB=\cC_e$ then the $G$-crossed braiding $c_{Z,X}:Z\otimes X\to X\otimes Z^g$ for $X\in\cC_g$ restricts to a relative braiding for $X\in\cC_e=\cB$. 
The other direction follows from \cite[Proposition 8.11]{DN}: to construct a $G$-crossed structure on $G$-graded extension it is enough to construct a central extension.  Here we include a rough explanation of how to extract a $G$-crossed structure from the structure of a central extension, as shown in the proof of {\it loc. cit.}\\

{Suppose that we have a $G$-graded central extension $\cB\subset \cC$.  The central functor $i: \mathcal{B} \to \mathcal{Z}(\cC)$ that sends $X \mapsto (X,c_{-,X})$ gives natural isomorphisms $c_{Z,X}:  Z \otimes X\overset{\sim}{\to}X \otimes Z $
for $X \in \mathcal{B}, Z \in \mathcal{C}$ satisfying hexagon axioms. In particular, each $\cC_g$ is an invertible left $\mathcal{B}$-module category with left action $Y_e \triangleright X_g := Y_e \otimes X_g$ and module associators inherited from $\cC$. 

For every pair $g,h\in G$ there are $\cB$-module equivalences $R_{g,h}$ and $L_{g,h}$ given by

\begin{align*} R_{g,h} : \, & \cC_h \to \Fun_{\mathcal{B}}(\cC_g, \cC_{gh}) \quad & \quad  L_{g,h} : \, & \cC_h \to \Fun_{\mathcal{B}}(\cC_g, \cC_{hg}) \\
&Y_h \mapsto [M_g \mapsto M_g \otimes Y_h] \quad & \quad &Y_h \mapsto [M_g \mapsto Y_h \otimes M_g] \\
\end{align*}
whose structure as $\mathcal{B}$-module functors are given by the natural isomorphisms

\begin{minipage}{.5\textwidth}
\begin{tikzcd}[scale=.8]
    R_{g,h}(X_e \otimes Y_h)(M_g)= M_g \otimes (X_e \otimes Y_h)   
    \ar[d, "\alpha^{-1}_{M_g,X_e,Y_h}"] 
    \\
     (M_g \otimes X_e) \otimes Y_h
    \ar[d, "c_{M_g,X_e} \otimes \id"] 
    \\
    (X_e \otimes M_g) \otimes Y_h 
    \ar[d, "\alpha_{X_e,M_g, Y_h}"] 
    \\
   X_e \otimes R_{g,h}(Y_h)(M_g) = X_e \otimes (M_g \otimes Y_h) 
\end{tikzcd}
\end{minipage}%
\begin{minipage}{.5\textwidth}
\begin{tikzcd}
    L_{g,h}(X_e \otimes Y_h)(M_g)= (X_e \otimes Y_h) \otimes M_g   
    \ar[d, "\alpha_{X_e,Y_h,M_g}"] 
    \\
    X_e \otimes L_{g,h}(Y_h)(M_g) = X_e \otimes (Y_h \otimes M_g).
    \end{tikzcd} \vspace{85pt}
    \end{minipage}\\

The $G$-action and $G$-braiding are given by a family of functors and natural isomorphism

\begin{align*}
 \left\{T_g: \cC_h \to \cC_{g^{-1}hg}, \  \ \ \  \ \ c_{Y_h,M_g}:Y_h\otimes M_g \to M_g\otimes T_g(Y_h)\right\}_{g,h}  
\end{align*}
defined by the following pasting diagram of $\cB$-module equivalences, which commutes up to a natural isomorphism $c$ that characterizes the $G$-action and $G$-braiding on $\cC$ up to equivalence. 
    \begin{equation}
        \label{eq:gactionfunctor}
    \begin{tikzcd}[execute at end picture={
   \node[rotate=-135] at (-.5,.25) {\small $\implies$};
   \draw (-.25,.25) node[right] {$c$};}]
    \cC_h \ar[r, "T_g"]  \ar[d, "L_{g,h}",swap] &  \cC_{g^{-1}hg} \ar[dl, "R_{g,g^{-1}hg}"]
    \\
    \Fun_{\mathcal{B}}(\cC_g, \cC_{hg}) & 
    \end{tikzcd}\end{equation}

\subsubsection{Zesting central functors}

\begin{lemma}
\label{lem:zestedcentralfunctor}
Let $\cC$ be a $G$-crossed braided category with $G$-brading $c$ and $(\lambda,\nu)$ a $G$-crossed associative zesting datum. The fusion category $\cC^{(\lambda,\nu)}$ constructed in Proposition \ref{prop:associativegzesting} has the structure of a central $G$-extension of $\cC_e$ with relative braiding $c_{X_g,Y_e}$ for all $X_g\in \cC$ and $Y_e\in \cC_e$.
\end{lemma}
\begin{proof}
Without loss of generality, we may assume $\cC$ is strict so that $T_e$ is the identity as a tensor functor and $\gamma_{e,e}$ the identity. Hence equations (5) and (6) in Definition \ref{def:gcrossedbraided} directly imply that $c_{X_g,Y_e}$ defines a relative braiding. 
\end{proof}

Now let $\cC$ be a (right) $G$-crossed braided category with trivial component $\cB$ and $\cC^{(\lambda,\nu)}$ be a $G$-crossed associative zesting. By Lemma \ref{lem:zestedcentralfunctor} the tensor category $\cC^{(\lambda,\nu)}$ has a canonical central extension. Hence, we may apply the result of Section \ref{subsection: from central G-ext to G-crossed ext} to describe the $G$-crossed braided structure on $\cC^{(\lambda,\nu)}$.

We will denote by $R^{{(\lambda,\nu)}}_{g,h}$ and $L^{{(\lambda,\nu)}}_{g,h}$ the $\cB$-module functors induced by right and left multiplication in $\cC^{(\lambda,\nu)}$, which now take the form
    
\begin{align*} R^{{(\lambda,\nu)}}_{g,h} : \, & \cC^{{(\lambda,\nu)}}_h \to \Fun_{\mathcal{B}}(\cC_g^{{(\lambda,\nu)}}, \cC_{gh}^{{(\lambda,\nu)}}) \quad & \quad  L_{g,h} : \, & \cC_h^{{(\lambda,\nu)}} \to \Fun_{\mathcal{B}}(\cC_g^{{(\lambda,\nu)}}, \cC_{hg}^{{(\lambda,\nu)}}) \\
&Y_h \mapsto [M_g \mapsto M_g \otimes Y_h \otimes \lambda(g,h)] \quad & \quad &Y_h \mapsto [M_g \mapsto Y_h \otimes M_g \otimes \lambda(h,g)].
\end{align*}

The following result is stated for strict $\cC$ but applies in general after inserting the appropriate associators in the diagrams. 

\begin{theorem} 
\label{thm:gcrossedzesting} Let $\cC$ be a (strict) $G$-crossed braided category and $\cC^{(\lambda,\nu)}$ be an associative $G$-crossed zesting. Then the $G$-crossed braided structure induced by the central structure of Lemma \ref{lem:zestedcentralfunctor} is given by

\begin{description}
\item[$G$-action functors]  \begin{align*}
 T_{g}^{{(\lambda,\nu)}}: \, & \cC^{{(\lambda,\nu)}}_h \to \cC^{{(\lambda,\nu)}}_{g^{-1}hg} \\  
 & Y_h \mapsto Y_h^g \otimes \lambda (h,g) \otimes \lambda(g,g^{-1}hg)^*
       \end{align*}
       
\item[$G$-braiding] 
       \begin{align}
\label{eq:zestedgcrossedbraiding}
    c^{{(\lambda,\nu)}}_{Y_h,M_g} &= 
    \begin{tikzpicture}[baseline=75*.5,line width=1, scale=.5]
 \draw(0,0)--(0,5.5);
   \begin{scope}[yscale=3.675,xscale=2,xshift=-3cm,yshift=1.5cm]
  \braid[number of strands=2] a_1;
  \end{scope}
     \draw (-4,0) node[below] {$\parcen{M_g}$};
     \draw (-2,0) node[below] {$\parcen{Y_h^g}$};
   \draw (0,0) node[below] {$(h,g)$};
   \draw[looseness=2] (3,0) node[below] {$(g,g^{-1}hg)^*$} to [out =90,in=90] (7,0) node[below] {$(g,g^{-1}hg)$};
   \draw (-4,5.5) node[above] {$\parcen{Y_h}$};
     \draw (-2,5.5) node[above] {$\parcen{M_g}$};
   \draw (0,5.5) node[above] {$(h,g)$};
  \end{tikzpicture}.
\end{align}

\item[$G$-action tensorators] 
\begin{align} \label{eq:zestedtensorator}\mu_{h^{\lambda}}^{X,Y} & =
\begin{tikzpicture}[line width=1,scale=.5,yscale=-1,baseline =-100]
\draw (-4,0) node[above] {$\parcen{X_k^h}$} -- (-4,14)node[below] {$\parcen{X_k^h}$};
\draw (2,0) node[above] {$\parcen{Y_g^h}$} -- (2,14)node[below] {$\parcen{Y_g^h}$};
\draw (9,14) node[below] {\small $(k^h,g^h)$};
 \draw (10.5,0)--(10.5,10.5);
\draw[looseness=1.5] (8,0)--(8,1.5) \br (4,5.5)--(4,14);
\draw (3.75,14) node[below] {\small $\phantom{^h}(g,h)\phantom{^h}$};
\draw (7.5,0) node[above] {\small $(k,g)^h$};
\draw (10.5,0) node[above] {\small $(kg,h)$};
\draw[looseness=1.5] (8,5.5) to [out=-90, in=-90](5.5,5.5)--(5.5,14);
\draw (6,14)node[below] {\small $(h,g^h)^*$};
 \draw[looseness=2] (13,0) -- (13,13) to [out=90, in=90] (10.5,13)--(10.5,10.5);
 \draw (13.75,0)node[above] {\small $(h,(kg)^h)^*$};
\draw[white, line width=10, looseness=.75] (8,6.5) \br (-2,10)--(-2,14);
\draw[looseness=.75] (8,6.5) \br (-2,10)--(-2,14);
\draw (-2.25,14) node[below] {\small $\phantom{^h}(k,h)\phantom{^h}$};
\draw[white, line width=10,looseness=1] (8,11.5) to [out=-90,in=-90] (0,11.5)--(0,14);
\draw[looseness=1.] (8,14)--(8,11.5) to [out=-90,in=-90] (0,11.5)--(0,14) node[below] {\small $(h,k^h)^*$};
\draw[fill=white] (6.5,5.5) rectangle node {$k,h,g^h$} (11.5,6.5);
\draw (11.5,5.5) node[right] {\tiny$-1$};
\draw[fill=white] (6.5,.5) rectangle node {$k,g,h$} (11.5,1.5);
\draw[fill=white] (6.5,11.5) rectangle node {$h, k^h, g^h$} (11.5,12.5);
\end{tikzpicture}
\end{align}

\item[$G$-action compositors]
\begin{align}
\label{eq:zestedcompositor}\gamma_{g^{\lambda},h^{\lambda}}^{X}= 
\begin{tikzpicture}[line width=1,scale=.5,yscale=-1, baseline=-80]
\draw (-.25,0)node[above] {$X_k^{gh}$} --(-.25,12) node[below] {$X_k^{gh}$};
\draw [looseness=2](14,0)--(14,11) to [out=90, in=90] (16,11)--(16,0);
\draw (13.5,0) node[above] {$(k,gh)$};
\draw (17,0) node[above] { $(gh, k^{gh})^*$};
\draw[looseness=2] (10,10.5) --(10,11) to [out=90,in=90] (12,11)--(12,10.75);
\draw (12,9.25)-- (12,7.75);
\draw (12,6.25)-- (12,4.75);
\draw[looseness=1] (2,3.5) -- (2,3) to [out=-90,in=-90] (12,3)--(12,3.25);
\draw (2,4) --(2,12);
\draw (1.5,12) node[below] {\small $(k,g)^h$};
\draw (4,6.5)--(4,12)node[below] {\small $(g,k^g)^{h*}$};
\draw[looseness=2] (4,6.5) to [out=-90,in=-90] (6,6.5)--(6,12);
\draw (6.5,12) node[below] {\small $(k^g,h)$};
\draw[fill=white] (5.5,6.5) rectangle node {\small $g,k^g,h$} (14.5,7.5);
\draw (14.5,9.75) node[right] {\tiny$-1$};
\draw (14.5,3.5) node[right] {\tiny$-1$};
\draw[looseness=2] (8,12) --(8,9.5) to [out=-90,in=-90] (10,9.5);
\draw (9.24,12) node[below] {\small $(h,k^{gh})^*$}; 
\draw[fill=white] (9.5,9.5) rectangle node {\small $g,h,k^{gh}$} (14.5,10.5);
\draw[fill=white] (1.5,3.5) rectangle node {\small $k,g,h$} (14.5,4.5);
\end{tikzpicture}.
\end{align}
\end{description}
\end{theorem}

\begin{proof}
We will show that the functor $T^{{(\lambda,\nu)}}_g$ makes the diagram  \ref{eq:gactionfunctor} commute up to the proposed natural isomorphism. Consider the following composition

\begin{equation}\label{cd:gcrossedaction}   \begin{tikzcd}
    L_{g,h}^{{(\lambda,\nu)}}(Y_h)(M_g)=  Y_h \otimes M_g \otimes \lambda(h,g) \phantom{L_{g,h}^{\lambda}(Y_h)(M_g)}
     \ar[d,"c_{Y_h,M_g} \otimes \id_{\lambda(h,g)}"]  \\ 
      M_g \otimes Y_h^g \otimes \lambda(h,g)  \ar[d,"\id_{M_g \otimes Y_h^g \otimes \lambda(h,g)} \otimes\, \text{coev}'_{\lambda(g,g^{-1}hg)} "]   \\
        M_g \otimes Y_h^g \otimes \lambda(h,g) \otimes \lambda(g,g^{-1}hg)^* \otimes \lambda(g,g^{-1}hg)   \ar[d,equal] \\
     \phantom{R_{hg,g^{-1}hg}^{(\lambda,\nu)}(Y_h^{g^{\lambda}})(M_g)} R_{g,g^{-1}hg}(T^{(\lambda,\nu)}(Y_h))(M_g)=M_g \otimes Y_h^{g^{\lambda}} \otimes \lambda(g,g^{-1}hg)\\
   \end{tikzcd}
    \end{equation} 
which gives a natural isomorphism from $L_{g,h}^{(\lambda,\nu)}$ to $R_{hg,g^{-1}hg}^{(\lambda,\nu)}\circ T^{(\lambda,\nu)}_g$. We can conclude that the zested $G$-action of $T_g^{{(\lambda,\nu)}}$ on objects is given by \begin{align}\label{eq:zestedgaction}
Y_h^{g^{\lambda}}:= Y_h^g \otimes \lambda(h,g) \otimes \lambda(g,g^{-1}hg)^*,
\end{align}
where we write $g^{\lambda}$ in the superscript to indicate that $g$ is the zested action by $T_g^{(\lambda,\nu)}$. 

    Finally we compute the natural isomorphisms
    $(\mu_{g^{\lambda}})_{X,Y}: ( X_h \overset{\lambda}{\otimes} Y_k)^{g^{\lambda}} \simeq  X_h^{g^{\lambda}}  \overset{\lambda}{\otimes}  Y_{k}^{g^{\lambda}}$ for all $X_h \in \cC_h, Y_k \in \cC_k$ and $(\gamma_{g^{\lambda},h^{\lambda}})_{X}: X_k^{(gh)^{\lambda}}  \simeq (X_k^{g^{\lambda}})^{h^{\lambda}}$ for all $X_k \in \cC_k$. The key observation is that the $G$-crossed braiding and the functors $T^{(\lambda,\nu)}$ already implicitly determine the isomorphisms $\mu_{g^{\lambda}}$ and $\gamma_{g^{\lambda},h^{\lambda}}$ through the $G$-crossed hexagon axioms (5) and (6) in Definition \ref{def:gcrossedbraided} via the formulas

\begin{align}
    \label{eq:zestedtensorators}
  \id_Z \overset{\lambda}{\otimes} (\mu_{h^{\lambda}})_{X,Y}= &  
 \alpha^{(\lambda,\nu)}_{Z,X^{h^{\lambda}},Y^{h^{\lambda}}} \circ c^{(\lambda,\nu)}_{X,Z} \overset{\lambda}{\otimes} \id_{Y^{h^{\lambda}}} \circ (\alpha^{(\lambda,\nu)}_{X,Z,Y^{h^{\lambda}}})^{-1}  \circ \id_X \overset{\lambda}{\otimes} c^{(\lambda,\nu)}_{Y,Z} \circ \alpha^{(\lambda,\nu)}_{X,Y,Z} \circ (c^{(\lambda,\nu)}_{X\overset{\lambda}{\otimes} Y,Z})^{-1}
\end{align}

\begin{align}
 \label{eq:zestedcompositors}
\id_{Y \overset{\lambda}{\otimes} Z} \overset{\lambda}{\otimes} (\gamma_{g^{\lambda},h^{\lambda}})_X =&  
 (\alpha^{(\lambda,\nu)}_{Y,Z,(X^{g^{\lambda}})^{h^{\lambda}}})^{-1}
\circ \id_Y \overset{\lambda}{\otimes}  c^{(\lambda,\nu)}_{X^{g^{\lambda}},Z}
\circ \alpha^{(\lambda,\nu)}_{Y,X^{g^{\lambda}},Z} 
\circ c^{(\lambda,\nu)}_{X,Y} \overset{\lambda}{\otimes} \id_Z 
\circ (\alpha^{(\lambda,\nu)}_{X,Y,Z})^{-1} \circ (c^{(\lambda,\nu)}_{X,Y \overset{\lambda}{\otimes} Z } )^{-1}.
\end{align}

By Remark \ref{rmk:technical} we can trace out the invertibles on the right-hand side of these equations introduced by the zested tensor product $\overset{\lambda}{\ot}$. Then one is left with equations for $\mu_{g^{\lambda}}$ and $\gamma_{g^{\lambda},h^{\lambda}}$ up to the original tensor product with the identity isomorphisms on $Z$ and $Y \overset{\lambda}{\otimes}$, respectively. Explicit diagrammatic formulas for $\mu_{g^{\lambda}}$ and $\gamma_{g^{\lambda},h^{\lambda}}$ are found by substituting the diagrams for the zested associators $\alpha^{(\lambda,\nu)}$ and $G$-crossed braiding $c^{(\lambda,\nu)}$ from Equations \ref{eq:zestedassociator} and \ref{eq:zestedbraiding} into Equations   \ref{eq:zestedcompositors} and \ref{eq:zestedtensorators} and simplifying as described.

Of course, when $\cC$ is pivotal this is immediate after tracing out Equations \ref{eq:zestedcompositor} and \ref{eq:zestedtensorator} (Section \ref{sec:zestedpivotality} discusses the pivotal structure this induces on $\cC^{(\lambda,\nu)}$). But the discussion above shows why pivotality of $\cC$ is not necessary in general.
\end{proof}
    
\begin{remark}
     In light of Theorem \ref{thm:gcrossedzesting}  an associative $G$-crossed zesting $\cC^{(\lambda,\nu)}$  automatically has the structure of a $G$-crossed braided fusion category. Thus, in analogy with braided zesting \cite[Definition 4.1]{DGPRZ}, we may make an {\it a posteriori}  definition of a {\it $G$-crossed braided zesting} as such a $\cC^{(\lambda,\nu)}$ obtained this way. Similarly, this justifies referring to an associative $G$-crossed zesting datum $(\lambda,\nu)$ as a {\it $G$-crossed braided zesting datum}.
 \end{remark}   
  
\begin{corollary}
Let $\mathcal{B}$ be a $G$-graded braided fusion category and $(\lambda, \nu)$ an associative (not $G$-crossed associative) zesting. Then $\mathcal{B}^{(\lambda,\nu)}$ is automatically $G$-crossed braided.
\end{corollary}

\begin{proof} If $(\lambda,\nu)$ is a reverse (see Remark \ref{rem:reverse zesting}) associative zesting data for $\cB$ then taking the $G$-action to be trivial as in Example \ref{ex: braided as G-crossed}, we obtain a $G$-crossed zesting data $(\lambda,\nu)$ for $\cB$.  The resulting category $\cB^{(\lambda,\nu)}$ is $G$-crossed braided by Theorem \ref{thm:gcrossedzesting}.  
\end{proof}
    
\begin{example}
Consider $\Vect_G$ with $G$-crossed braided structure given as in Example \ref{ex:vecgomega} (with trivial $3$-cocycle). Now, if we let $\lambda$ be trivial in the $G$-crossed zesting construction then any $3$-cocycle $\omega\in H^3(G,\Bbbk^{\times})$ gives a $G$-crossed zesting datum: $(1,\omega)$.  The resulting $G$-crossed braided fusion category is $\Vect_G^\omega$ as described in Example \ref{ex:vecgomega}.
\end{example}    
   
\subsubsection{Connection between \texorpdfstring{$G$}{G}-crossed zesting and \texorpdfstring{$G$}{G}-extension theory}

In \cite{DN} $G$-crossed braided extensions of a braided tensor category $\cB$ were classified by means of the Picard 2-categorical
group $\textbf{Pic}(\cB)$ of invertible $\cB$-module categories. More concretely, Davydov and Nikshych have shown that $G$-crossed braided extensions correspond to monoidal 2-functors $\underline{\underline{G}} \to \textbf{Pic}(\cB)$. In consequence, equivalence classes of extensions can be described in terms of certain cohomology groups associated with a group morphism $G \to \Pic{\cB}$, where $\Pic{\cB}$ is the group of equivalence classes of invertible $\cB$-module categories. For more details see {\it loc. cit.}

\begin{theorem}\label{thm:extensions}
Let $\mathcal{B}$ be a braided fusion category. Two $G$-crossed braided fusion categories have the same group morphism $G\to \Pic{\cB}$ if and only if they are related by $G$-crossed braided zesting.
\end{theorem}
\begin{proof}

Proposition 8.11 and Theorem 8.13 in \cite{DN} show that there is an equivalence between equivalence classes of $G$-crossed braided extensions of $\mathcal{B}$ and categorical $2$-group homomorphisms from $\underline{\underline{G}}$ to $\textbf{Pic}(\cB)$. 

Let $\cC$ be a $G$-crossed braided extension of $\cB$ and $\cC^{(\lambda,\nu)}$ a $G$-crossed braided zesting of $\cC$. Since the trivial component remains the same under zesting, the normalization conditions imposed on $\lambda$ and $\nu$ force  $\cC^{(\lambda, \nu)}_g=\cC_g$ as $\cB$-module categories for all $g\in G$. In fact, it follows from the normalization condition in Definition \ref{def:gcrossedassociativezesting} that the zested module associators are equal to the original ones. Hence, the group morphism $G\to \Pic{\cB}$ associated to $\cC^{(\lambda,\nu)}$  is the same one. 

On the other hand, consider $(\cC,\otimes)$ and $(\cC',\otimes')$ $G$-crossed braided extensions of $\cB$ such that $\cC_g\cong \cC_g'$ for all $g\in G$ as $\cB$-module categories. In particular, $(\cC,\otimes)$ and $(\cC',\otimes')$ are defining the same group homomorphism $G\to \Pic{\cB}$. 

It follows from \cite{DN} and \cite{ENO} that all possible liftings of a fixed group homomorphism $G\to \Pic{\cB}$ to $\underline{G} \to \Pi_{\leq 1}(\textbf{Pic}(\cB))$ form a torsor over $H^2(G,\Inv(\cB))$, where $\Pi_{\leq 1}(\textbf{Pic}(\cB))$ is the categorical group obtained by truncating the categorical 2-group $\textbf{Pic}(\cB)$. Then there exists a 2-cocycle $\lambda:G\times G\to \Inv(\cB)$ such that $X_g\otimes'Y_h=X_g\otimes Y_h\otimes \lambda(g,h)$. Since $\cC'$ is a $\cB$-extension, the $PW(\lambda)\in H^4(G,\ku^{\times})$ vanishes \cite[Definition 8.11]{ENO}. By Remark \ref{remark H4 obstruction}, there exists a zesting datum $(\lambda,\nu)$ in such a way that $\cC^{(\lambda,\nu)}$ and $\cC'$ determine the same categorical group homomorphism $\underline{G}\to \Pi_{\leq 1}(\textbf{Pic}(\cB))$. Again, all possible liftings of such a fixed categorical group homomorphism to $G \to \Pic{\cB}$ form a torsor over $H^3(G,\ku^{\times})$. Thus, there exists a 3-cocycle $\omega:G\times G\times G\to \ku^{\times}$ 
such that $\cC\cong (\cC^{(\lambda,\nu)})^\omega$ as $G$-crossed braided extensions, where $(\cC^{(\lambda,\nu)})^\omega$ corresponds to 
twisting the associativity constraint by the 3-cocycle $\omega$. Finally, note $(\lambda,\nu\omega)$ is a zesting datum and $(\cC^{(\lambda,\nu)})^\omega=\cC^{(\lambda,\nu\omega)}$. Then $\cC'$ is equivalent to a zesting of $\cC$ with zesting datum $(\lambda,\nu\omega)$.
\end{proof}

\subsection{\texorpdfstring{$\mathds{Z}/N\mathds{Z}$}{Z/NZ}-crossed zesting}\label{cyclic zesting}

Before describing cyclic-crossed zesting we need to recall a well-known fact about the cohomology of cyclic groups.

Let $C_m$ be the cyclic group of order $m$ generated by $\sigma$ and $A$ a $C_m$-module. The sequence
$$\begin{diagram}\node{\cdots}\arrow{e,t}{N} \node{\mathbb Z
C_m}\arrow{e,t}{\sigma -1}\node{\mathbb Z
C_m}\arrow{e,t}{N}\node{\mathbb Z C_m}\arrow{e,t}{\sigma
-1}\node{\mathbb Z C_m}\arrow{e}\node{\mathbb Z}\end{diagram}$$
where  $N= 1+ \sigma + \sigma^2 +\cdots +\sigma ^{m-1}$ is a free resolution of $\mathbb Z$ as $C_m$-module. Using that $H^n(C_m,A)= \operatorname{Ext}_{\mathbb{Z}[C_m]}(\Z, A)$, we have that

\begin{align}\label{eq:coho_cyclic} H^n(C_m, A)= \begin{cases}\ker(N)/\operatorname{Im}(1-\sigma), \qquad &\text{if } n=1, 3, 5, \ldots \\
\Ker(1-\sigma)/\operatorname{Im}(N), \quad  &\text{ if } n = 2, 4, 6, \ldots.
\end{cases}
\end{align}

\label{sec:cyclicGcrossedzesting}

Let $\cC$ be an $C_m$-crossed braided fusion category with  $\mathcal{T}$ the maximal pointed fusion subcategory of the trivial component $\cC_e$.

Hence the abelian group  $\Inv(\mathcal{T})$ is a $C_m$-module and $$H^2(C_m,\Inv(\mathcal{T}))= \mathcal{T}^{C_m}/\operatorname{Im}(N).$$

An explicit isomorphism between $H^2(A,\Inv(\mathcal{T}))$ and $ \mathcal{T}^{C_m}/\operatorname{Im}(N)$ is defined via the 2-cocycles:

\begin{align}\label{cyclic2cocycle}
\lambda_g(\sigma^i,\sigma^j) = \begin{cases} \mathds{1}, & i+j < N \\
g, & i+j \ge N \end{cases} 
\end{align}
where $0\leq i, j < m$ and $g\in \mathcal{T}^{\Z/N\Z}$. 

Fix such a $\lambda_g$. It follows from \ref{eq:coho_cyclic} that  $H^4(C_m,\Bbbk^\times)=0$ so the PW function (\ref{defintion of PW}) automatically vanishes for any $\lambda_g$.  Thus there is no obstruction to lifting $\lambda_g$ and there exists a $C_m \cong H^3(C_m,\Bbbk^\times)$-torsor's worth of cochains $\nu_b\in C^3(C_m,\Bbbk^\times)$ so that $(\lambda_a,\nu_b)$ for $(a,b)\in \mathcal{T}^{C_m}/\operatorname{Im}(N)\times C_m$ form an $C_m$-crossed associative zesting datum for $\cC$. Given a particular $\nu$ satisfying (\ref{eq:associativeGzesting}) we may parameterize all such cochains as $\nu_b=\nu\cdot \xi_b$
where $q=e^{2\pi i/N}$ and
\begin{align}\label{cyclic3cochain}
\xi_b(\sigma^i,\sigma^j,\sigma ^k) = \begin{cases} 1 & i+j < N \\
q^{bk} & i+j \ge N \end{cases},
\end{align}where $0\leq i,j,k<m$.

The $C_m$-crossed braided structure of $\cC^{(\lambda_g,\nu_b)}$ is described as follows.  The action of $\sigma^j\in C_m$ on $X_{\sigma^i} \in \cC^{(\lambda_g,\nu_b)}_{\sigma^i}$ is given by
\[
X_{\sigma^i}^{\sigma^{j^{\lambda}}} := X_{\sigma^i}^{\sigma^j} \otimes \lambda_g(i,j) \otimes \lambda_g(j,i)^* \cong X_{\sigma i}^{\sigma^j}.
\]
Hence up to a canonical isomorphism, the zested action on the objects is equal to the original action. However, the higher data of the zested action is nontrivial in general.  Since we are assuming that the $C_m$-action is strict, 
for a fixed $\nu_b$ it is possible to identify both $(\mu_b)^{X_i,X_j}_{k^{\lambda}}$ and $(\gamma_b)^{X_k}_{i^{\lambda},j^{\lambda}}$ with scalars $(\tilde{\mu}_b)^{X_i,Y_i}_k, (\tilde{\gamma}_b)_{i,j}^{X_k} \in \mathds{k}^\times$ by projecting onto the one-dimensional subspaces spanned by a chosen vector, namely 

$$\id_{X_i^k} \ot \text{coeval}_{\lambda(i,k)} \ot \id_{X_j^k} \ot \text{coeval}_{\lambda(j,k)} \ot \id_{\lambda(i,j)} \ot \text{eval}'_{\lambda(ij,k)} \in \Hom((X_i \overset{\lambda}{\otimes} X_j)^{k^{\lambda}},X_i^{k^{\lambda}} \overset{\lambda}{\otimes} X_j^{k^{\lambda}})$$
and$$\id_{X_k^{ij}} \ot \text{coeval}_{\lambda(i,k)} \ot \text{coeval}_{\lambda(j,k)} \ot \text{eval}'_{\lambda(ij,k)} \in \Hom(X_k^{(ij)^{\lambda}}, (X_k^{i^{\lambda}})^{j^{\lambda}}),$$ respectively. 

The scalars can be computed by applying the graphical calculus to the following equations.
\begin{align} 
\begin{tikzpicture}[line width=1,scale=.4,yscale=-1,baseline =-80]
\draw (-4,0) node[above] {$\parcen{X_i^k}$} -- (-4,14)node[below] {$\parcen{X_i^k}$};
\draw (2,0) node[above] {$\parcen{X_j^k}$} -- (2,14)node[below] {$\parcen{X_j^k}$};
\draw (9,14) node[below] {\small $(i,j)$};
 \draw (10.5,0)--(10.5,10.5);
\draw[looseness=1.5] (8,0)--(8,1.5) \br (4,5.5)--(4,14);
\draw (3.75,14) node[below] {\small $\phantom{^*}(j,k)\phantom{^*}$};
\draw (7.5,0) node[above] {\small $(i,j)$};
\draw (10.5,0) node[above] {\small $(ij,k)$};
\draw[looseness=1.5] (8,5.5) to [out=-90, in=-90](5.5,5.5)--(5.5,14);
\draw (6,14)node[below] {\small $(j,k)^*$};
 \draw[looseness=2] (13,0) -- (13,13) to [out=90, in=90] (10.5,13)--(10.5,10.5);
 \draw (13.75,0)node[above] {\small $(ij,k)^*$};
\draw[white, line width=10, looseness=.75] (8,6.5) \br (-2,10)--(-2,14);
\draw[looseness=.75] (8,6.5) \br (-2,10)--(-2,14);
\draw (-2.25,14) node[below] {\small $\phantom{^*}(i,k)\phantom{^*}$};
\draw[white, line width=10,looseness=1] (8,11.5) to [out=-90,in=-90] (0,11.5)--(0,14);
\draw[looseness=1.] (8,14)--(8,11.5) to [out=-90,in=-90] (0,11.5)--(0,14) node[below] {\small $(i,k)^*$};
\draw[fill=white] (6.5,5.5) rectangle node {$\scriptstyle i,k,j$} (11.5,6.5);
\draw (11.5,5.5) node[right] {\tiny$-1$};
\draw[fill=white] (6.5,.5) rectangle node {$\scriptstyle i,j,k$} (11.5,1.5);
\draw[fill=white] (6.5,11.5) rectangle node {$\scriptstyle k, i, j$} (11.5,12.5);
\end{tikzpicture}
= (\tilde{\mu}_b)^{X_i,Y_i}_k
\begin{tikzpicture}[line width=1,scale=.4,yscale=-1,baseline =-80]
\draw (-4,0) node[above] {$\parcen{X_i^k}$} -- (-4,14)node[below] {$\parcen{X_i^k}$};
\draw (2,0) node[above] {$\parcen{X_j^k}$} -- (2,14)node[below] {$\parcen{X_j^k}$};
\draw[looseness=2] (6,14) to [out=-90, in=-90] (4,14);
\draw (3.75,14) node[below] {\small $\phantom{^*}(j,k)\phantom{^*}$};
\draw (8.5,14) node[below] {\small $(i,j)$} -- (8.5,0) node[above] {\small $(i,j)$};
\draw (11,0) node[above] {\small $(ij,k)$};
\draw (6,14)node[below] {\small $(j,k)^*$};
\draw[looseness=2] (13,0) to [out=90, in=90] (11,0);
\draw (13.5,0) node[above] {\small $(ij,k)^*$};
\draw (-2.25,14) node[below] {\small $\phantom{^*}(i,k)\phantom{^*}$};
\draw[looseness=2] (-2,14) to [out=-90,in=-90] (0,14) node[below] {\small $(i,k)^*$};
\end{tikzpicture}
\end{align}

\begin{align}
\begin{tikzpicture}[line width=1,scale=.4,yscale=-1, baseline=-60]
\draw (-.25,0)node[above] {$X_k^{ij}$} --(-.25,12) node[below] {$X_k^{ij}$};
\draw [looseness=2](14,0)--(14,11) to [out=90, in=90] (16,11)--(16,0);
\draw (13.5,0) node[above] {$(ij,k)$};
\draw (17,0) node[above] { $(ij, k)^*$};
\draw[looseness=2] (10,10.5) --(10,11) to [out=90,in=90] (12,11)--(12,10.75);
\draw (12,9.25)-- (12,7.75);
\draw (12,6.25)-- (12,4.75);
\draw[looseness=1] (2,3.5) -- (2,3) to [out=-90,in=-90] (12,3)--(12,3.25);
\draw (2,4) --(2,12);
\draw (1.5,12) node[below] {\small $(i,k)$};
\draw (4,6.5)--(4,12)node[below] {\small $(i,k)^{*}$};
\draw[looseness=2] (4,6.5) to [out=-90,in=-90] (6,6.5)--(6,12);
\draw (6.5,12) node[below] {\small $(j,k)$};
\draw[fill=white] (5.5,6.5) rectangle node {$\scriptstyle i,k,j$} (14.5,7.5);
\draw (14.5,9.75) node[right] {\tiny$-1$};
\draw (14.5,3.5) node[right] {\tiny$-1$};
\draw[looseness=2] (8,12) --(8,9.5) to [out=-90,in=-90] (10,9.5);
\draw (9.24,12) node[below] {\small $(j,k)^*$}; 
\draw[fill=white] (9.5,9.5) rectangle node {$\scriptstyle i,j,k$} (14.5,10.5);
\draw[fill=white] (1.5,3.5) rectangle node {$ \scriptstyle k,i,j$} (14.5,4.5);
\end{tikzpicture}
=
(\tilde{\gamma}_b)_{i,j}^{X_k} 
\begin{tikzpicture}[line width=1,scale=.4,yscale=-1, baseline=-60]
\draw (-.25,0)node[above] {$X_k^{ij}$} --(-.25,12) node[below] {$X_k^{ij}$};
\draw (10.5,0) node[above] {$(ij,k)$};
\draw (13.5,0) node[above] { $(ij, k)^*$};
\draw[looseness=2] (2,12) to [out=-90,in=-90] (4,12);
\draw[looseness=2] (6,12) to [out=-90,in=-90] (8,12);
\draw[looseness=2] (11,0) to [out=90,in=90] (13,0);
\draw (1.5,12) node[below] {\small $(i,k)$};
\draw (4,12)node[below] {\small $(i,k)^{*}$};
\draw (6.5,12) node[below] {\small $(j,k)$};
\draw (9.25,12) node[below] {\small $(j,k)^*$}; 
\end{tikzpicture}
\end{align}
Solving for $\tilde{\mu}$ and $\tilde{\gamma}$ we find
    \begin{align}
        (\tilde{\mu}_{b})_k^{X_{i},Y_{j}}= \frac{\nu_b(i,j,k)\nu_b(k,i,j)}{\nu_b(i,k,j)}=\nu_b(i,j,k) \\
    (\tilde{\gamma}_b)^{X_{k}}_{i,j}=\frac{\nu_b(j,i,k)}{\nu_b(i,j,k)\nu_b(j,i,k)} = \nu_b(i,j,k)^{-1}
    \end{align}
again using that $\lambda_g$ is symmetric and that $\nu_b$ is symmetric in its first two arguments. 

\begin{remark} Note that when $X_k$ is simple $$\Hom_{\cC^{(\lambda,\nu)}}(X_k^{(ij)^{\lambda}}, (X_k^{i^{\lambda}})^{j^{\lambda}}) \cong \Hom_{\cC}(X_k^{ij}, X_k^{ij}) \cong \mathds{k}$$ and it is natural to want to identify $\gamma$ with a scalar (namely a multiple of the identity morphism $\id_{X_k}$). Of course the same is not true for $\mu$: $\Hom_{\cC^{(\lambda,\nu)}}((X_i \overset{\lambda}{\otimes} X_j)^{k^{\lambda}},X_i^{k^{\lambda}} \overset{\lambda}{\otimes} X_j^{k^{\lambda}})$ is not one-dimensional in general for simple $X_i$ and $X_j$. In summary, we have done something slightly unusual if reasonably natural to analyze the action here, but we reiterate that $(\mu_b)^{X_i,X_j}_{k^{\lambda}}$ is not an isomorphism of one-dimensional vector spaces.
\end{remark}

The $C_m$-crossed braiding is simply
\begin{align}
c^{(\lambda_g,\nu_b)}_{X_{i},X_{j}} = c_{X_{i},X_{j}} \ot \id_{\lambda_g(i,j)}\otimes \text{coeval}'_{\lambda_g(i,j)}
\end{align} 
where $c$ is the $A$-crossed braiding in $\cC$.

Here is a special case, as a continuation of Example \ref{ex:su33}.
 \begin{example} 
 \label{ex:Gcrossedzestingsu33}
The associative $\Z/ 3\Z$-crossed zestings $(\lambda_a,\nu_b)$ of $SU(3)_3$ are parameterized by pairs $(a,b)\in \Z/3\Z \times \Z/3\Z$. Recall from Example \ref{ex:su33} that only three out of the nine associative zestings $SU(3)_3^{(\lambda_a,\nu_b)}$ admit a braiding. Viewing $SU(3)_3$ as a $\Z/ 3\Z$-crossed braided category with the trivial $\Z/ 3\Z$-action, we find that all nine fusion categories $SU(3)_3^{(\lambda_a,\nu_b)}$ are $\Z/3\Z$-crossed braided.

\end{example}               
\subsection{Recovering the braided zesting construction}
\label{sec:recoveringbraidedzesting}
Recall from Remark \ref{rmk:trivialization} that a $G$-crossed braided category $\cC$ admits a braiding if the $G$-action admits a trivialization, i.e. a monoidal natural isomorphism of $T$ with the identity functor $\Id_{\cC}$. In particular, $G$ must be abelian. Throughout this section, $G$ will always denote an abelian group. In slightly more detail:

\begin{definition}
\label{def:trivialization}
Let $\cC$ be a (right) spherical $G$-crossed braided tensor category with $G$-action $(T,\mu, \gamma)$. A trivialization $\eta$ of the $G$-action is a family of natural isomorphisms $\eta_g: T_g \to \Id_{\cC}$ satisfying
\begin{enumerate}
    \item $\eta_h \circ \eta_g^h = \eta_{gh} \circ \gamma_{g,h}$
    \item $\eta_{g}= (\eta_g \otimes \eta_g)\circ \mu_g$
    for all $g,h \in G$.
\end{enumerate}
\end{definition}
We draw the equations that define $\eta$ for all $X, Y \in \cC$ as 

\begin{minipage}{.5\textwidth}
\begin{equation}
\label{eq:trivializationcondition1}
\begin{tikzpicture}[line width=1,scale=.75,baseline=1.25cm]
\onecirc{X}{}{\eta_{gh}}
\begin{scope}[yshift=1.5cm, xscale=1.5]
\onebox{}{(X^g)^h}{\gamma_{g,h}}
\draw (.25,1.25) node[right] {$^{-1}$};
\end{scope}
\end{tikzpicture}=
\begin{tikzpicture}[line width=1,scale=.75,baseline=1.25cm]
\onecirc{X}{}{\eta_{h}}
\begin{scope}[yshift=1.5cm]
\onecirc{}{(X^g)^h}{\eta_g}
\draw (.25,1.25) node[right] {$^h$};
\end{scope}
\end{tikzpicture} \qquad\text{ and }\end{equation}
\end{minipage}%
\begin{minipage}{.5\textwidth}
\begin{equation}
\label{eq:trivializationcondition2}
\begin{tikzpicture}[line width=1,scale=.75,baseline=1.25cm]
\draw (0,0) node[below] {$X\otimes Y$}--(0,.75);
\begin{scope}[yshift=.75cm]
\onecirc{}{}{\eta_{g}}
\end{scope}
\draw (0,2.75)--(0,3.5) node[above] {$(X\otimes Y)^g$};
\end{tikzpicture}
=
\quad
\begin{tikzpicture}[line width=1,scale=.75,baseline=1.25cm]
\onecirc{X}{}{\eta_{g}}
\draw (.75,2.75)--(.75,3.5) node[above] {$(X\otimes Y)^g$};
\begin{scope}[xshift=1.5cm]
\onecirc{Y}{}{\eta_{g}}
\end{scope}
\begin{scope}[yshift=1.5cm,xscale=1.5]
\twoboxnotop{}{}{\mu_g}{}
\end{scope}
\end{tikzpicture}.
\end{equation}
\end{minipage}

We use rectangles for the structure morphisms of the $G$-action and circles for the trivialization $\eta$ to facilitate comparison with \cite[Section 4]{DGPRZ}. Note that we have suppressed the dependence of $\gamma_{g,h}, \mu_g, \eta_g$ on $X, Y \in \cC$ for simplicity. 

Next recall from Theorem 3.12 in \cite[Theorem 3.12]{Galindo2022} that if $\cC$ is a $G$-crossed braided monoidal category with trivialization $\eta$, the natural isomorphisms $c^{(\eta)}_{X_g,Y_h}: X \otimes Y_h \to Y_h \otimes X_g$ given by

\begin{align}
\label{eq:braidingfromtrivialization}
\begin{tikzpicture}[line width=1,scale=.75,baseline =-.75*.25cm]
\draw (-1,-.25) node[left] {$c^{(\eta)}_{X_g,Y_h}=$};
\draw[looseness = 1.5] (0,0) node[below]{} \br (1,1.5) node[above]{$Y_h$};
\draw[white, line width=10, looseness = 1.5] (1,0) node[below]{} \br (0,1.5) node[above]{$X_g$};
\draw[looseness = 1.5] (1,0) node[below]{} \br (0,1.5) node[above]{$X_g$};
\draw (0,-2) node[below] {$Y_h$} -- (0,0);
\begin{scope}[xshift=1cm, yshift=-2cm]
\onecirc{X_g}{}{\eta_h}
\end{scope}
\end{tikzpicture}
\text{ where }  \begin{tikzpicture}[line width=1,scale=.75,baseline=.75*.75cm]
\draw[looseness = 1.5] (0,0) node[below]{$Y_h$} \br (1,1.5) node[above]{$Y_h$};
\draw[white, line width=10, looseness = 1.5] (1,0) node[below]{$X_g^h$} \br (0,1.5) node[above]{$X_g$};
\draw[looseness = 1.5] (1,0) node[below]{$X_g^h$} \br (0,1.5) node[above]{$X_g$};
\end{tikzpicture}  \text{ is the $G$-crossed braiding on $\cC$}
\end{align}

defines a braiding on $\cC$ and establishes a biequivalence between the 2-category of $G$-crossed braided fusion categories with trivialization and the 2-category of $G$-graded braided tensor categories. 

Now we will see that when the subcategory of invertibles of a braided fusion category generated by $\lambda(a,b)$ for all $a,b \in G$ is symmetric, a braided zesting datum $(\lambda,\nu, t)$ can be recovered from a $G$-crossed zesting datum $(\lambda,\nu)$ and a trivialization $\eta$. 

\begin{theorem}
\label{thm:trivialization}
Let $\cB$ be a braided fusion category faithfully graded over a (necessarily abelian) group $G$ and suppose $\cB_e \cap \cB_{pt}$ is symmetric. Then a braided $G$-zesting datum $(\lambda, \nu, t)$ of $\cB$ is equivalent to that of a $G$-crossed braided zesting datum $(\lambda, \nu)$ together with a trivialization of the zested $G$-action $\eta$ on $\cB^{(\lambda,\nu)}$ and there is an equivalence of braided tensor categories $\cB^{(\lambda, \nu, t)} \simeq \cB^{(\lambda,\nu)}$. Moreover, the obstructions to braided zesting of $\cB$ are precisely the obstructions to defining a trivialization $\eta$ of the zested $G$-action on $\cB^{(\lambda,\nu)}$.
\end{theorem}

\begin{proof}
Begin by regarding a $G$-graded braided fusion category $\cB$ as a $G$-crossed braided fusion category equipped with the trivial action as in Example \ref{ex: braided as G-crossed}. If $\lambda \in Z^2(G,\Inv(\cB_e))$, then $\nu \in C^3(G,\mathds{k}^{\times})$ defines an associative zesting of $\cB$ if and only if $\nu$ defines a $G$-crossed associative zesting of $\cB$ as a $G$-crossed braided category. This follows from Remark \ref{rem:reverse zesting}. By Theorem \ref{thm:gcrossedzesting} $\cB^{(\lambda,\nu)}$ is a $G$-crossed braided tensor category with $G$-action $X_g^h = X_g \otimes \lambda(g,h) \otimes \lambda(h,g)^*$ and $\mu_g$, $\gamma_{g,h}$ as in Equations \ref{eq:zestedtensorator} and \ref{eq:zestedcompositor}.

Now we claim that there exists a braided zesting $\cB^{(\lambda,\nu,t)}$ if and only if the $G$-crossed braided category $\cB^{(\lambda,\nu)}$ admits a trivialization and hence a braiding.

Suppose $t_{g,h}:\lambda(g,h) \to \lambda(h,g)$ is a natural isomorphism for all $g, h \in G$ satisfying the braided zesting conditions and define isomorphisms $\eta^{(\lambda,t)}$
\begin{align*}
\begin{tikzpicture}[line width=1,scale=.75,baseline=1.25cm,looseness=2]
\draw (-.75,1) node[left] {$\eta^{(\lambda,t)}(X_g):=$};
\draw (-.25,0) node[below] {$X_g$}--(-.25,2) node[above] {$X_g$};
\draw (1.5,2) to [out=-90, in=-90] (3,2);
\draw[fill=white] (2.25,1) circle (.5) node {$t_{g,h}$};
\draw (1.25,2) node[above] {$\lambda(g,h)$};
\draw (3.25,2) node[above] {$\lambda(h,g)^*$};
\end{tikzpicture}
\end{align*}
for all $X_g \in \cB_g$.

Substituting in the pictures for $\eta$, $\mu$, and $\gamma$ one finds that the defining Equations \ref{eq:trivializationcondition1} and \ref{eq:trivializationcondition2} of a trivialization from Definition \ref{def:trivialization} hold only if the following diagrams are identities.
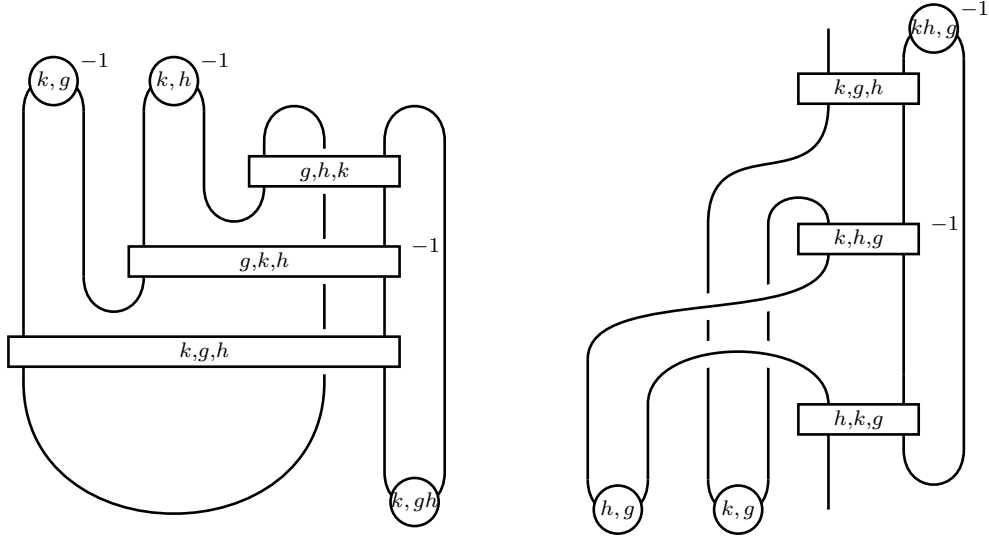
\begin{align*}
\begin{tikzpicture}[line width=1,scale=.4,yscale=1, baseline=-.5cm, looseness=2]
\draw [looseness=2](14,0)--(14,11) to [out=90, in=90] (16,11)--(16,0);
\draw[looseness=2] (10,10.5) --(10,11) to [out=90,in=90] (12,11)--(12,10.75);
\draw (12,9.25)-- (12,7.75);
\draw (12,6.25)-- (12,4.75);
\draw[looseness=1.5] (2,3.5) -- (2,3) to [out=-90,in=-90] (12,3)--(12,3.25);
\draw (2,4) --(2,12);
\draw (4,6.5)--(4,12);
\draw[looseness=2] (4,6.5) to [out=-90,in=-90] (6,6.5)--(6,12);
\draw[fill=white] (5.5,6.5) rectangle node {$\scriptstyle g,k,h$} (14.5,7.5);
\draw (14.5,7.5) node[right] {\tiny$-1$};
\draw[looseness=2] (8,12) --(8,9.5) to [out=-90,in=-90] (10,9.5);
\draw[fill=white] (9.5,9.5) rectangle node {$\scriptstyle g,h,k$} (14.5,10.5);
\draw[fill=white] (1.5,3.5) rectangle node {$\scriptstyle k,g,h$} (14.5,4.5);
\draw (2,12) to [out=90, in=90] (4,12);
\draw (6,12) to [out=90, in=90] (8,12);
\draw (14,0) to [out=-90, in=-90] (16,0);
\draw[fill=white] (3,13) circle (.8) node {\tiny  $k,g$};
\draw (3.5,13.5) node[right] {$^{-1}$};
\draw[fill=white] (7,13) circle (.8) node {\tiny  $k,h$};
\draw (7.5,13.5) node[right] {$^{-1}$};
\draw[fill=white] (15,-1) circle (.8) node {\tiny $k,gh$};
\end{tikzpicture} \qquad & \qquad  \begin{tikzpicture}[line width=1,scale=.4,yscale=-1,baseline =-6cm,looseness=2]
\draw (10.5,0)--(10.5,10.5);
\draw[looseness=1.5] (8,-1)--(8,1.5) \br (4,5.5)--(4,14);
\draw[looseness=1.5] (8,5.5) to [out=-90, in=-90](6,5.5)--(6,14);
\draw[looseness=2] (12.5,0) -- (12.5,13) to [out=90, in=90] (10.5,13)--(10.5,10.5);
\draw[white, line width=10, looseness=.75] (8,6.5) \br (0,10)--(0,14);
\draw[looseness=.75] (8,6.5) \br (0,10)--(0,14);
\draw[white, line width=10,looseness=1] (8,11.5) to [out=-90,in=-90] (2,11.5)--(2,14);
\draw[looseness=1] (8,15)--(8,11.5) to [out=-90,in=-90] (2,11.5)--(2,14);
\draw[fill=white] (7,5.5) rectangle node {$\scriptstyle k,h,g$} (11,6.5);
\draw (11.,5.5) node[right] {\tiny$-1$};
\draw[fill=white] (7,.5) rectangle node {$\scriptstyle k,g,h$} (11,1.5);
\draw[fill=white] (7,11.5) rectangle node {$\scriptstyle h, k, g$} (11,12.5);
\draw (10.5,0) to [out=-90,in=-90] (12.5,0);
\draw[fill=white] (11.5,-1) circle (.8) node {\tiny  $kh,g$};
\draw (12,-1.5) node[right] {$^{-1}$};
\draw (4,14) to [out=90,in=90] (6,14);
\draw[fill=white] (5,15) circle (.8) node {\tiny  $k,g$};
\draw (0,14) to [out=90,in=90] (2,14);
\draw[fill=white] (1,15) circle (.8) node {\tiny  $h,g$};
\end{tikzpicture}
\end{align*}

These diagrams are isotopic to (traces of) the obstructions to braided zesting from \cite[Figure 13 and 14]{DGPRZ} in the case where $j$ and $\chi$ are trivial. (Triviality of $\chi$ follows from the assumption that $\cB_e \cap \cB_{pt}$ is symmetric.) Thus $\eta^{(\lambda,t)}$ defines a trivialization of the $G$-action on $\cB^{(\lambda,\nu)}$ and $\cB^{(\lambda,\nu)}$ is braided.

The other direction is similar; one needs only to observe that a necessary and sufficient condition for $T^{(\lambda,\nu)}$ to admit a trivialization is the existence of natural isomorphisms $\lambda(g,h) \cong \lambda(h,g)$ for all $g,h \in G$. Then one can check that such isomorphisms satisfy the braided zesting conditions.
\end{proof}

In the special context of cyclic $A$-crossed zesting with values in a Tannakian subcategory $\Rep(A,1)\subset \cB_e$ one has the following explicit description of when a trivialization exists:
\begin{proposition}[Adapted from Proposition 6.3 and 6.4 \cite{DGPRZ}]
\label{prop:cyclicTannakianzesting}
Let $\mathcal{B}$ be an $A$-graded braided fusion category with $A=\langle g \rangle\cong \Z/N\Z$, and $\Rep(A,1)\subset \cB_e$ a Tannakian subcategory of the trivial component. Fix $q$ a primitive $N$th root of unity. The equivalence classes of braided zestings of $\mathcal{B}$ with $j=\id$ are parameterized by triples $(a,b,s)$ where $(a,b) \in \Z/N \Z \times \Z / N \Z$ with $a = 2b \mod N$ and $s^N=q^{-b}$. Explicit braided zesting data $(\lambda, \nu, t)$ for each admissible $(a,b)$ and choice of $s$ is given by (\ref{cyclic2cocycle}) and (\ref{cyclic3cochain}) as above and
\begin{align}
    t_{a,b}(i,j) = s^{-ij} \id_{\lambda_a(i,j)}
\end{align}
where $0 \le i,j,k < N$. 
\end{proposition}
\qed

Continuing with our running $SU(3)_3$ example (Exs. \ref{ex:su33},\ref{ex:Gcrossedzestingsu33} we have the following:
\begin{example}
    Out of the 9 $\Z/3\Z$-crossed associative zestings $SU(3)_3^{(\lambda_a,\nu_b)}$ there are 3 that admit trivializations, and hence (honest) braidings. There are 3 choices of braidings for each. Here we display the particulars for completeness.
    \[\begin{array}{l|c|c|l}
(a,b,s) & \lambda_a(i,j) & \nu_b(i,j,k) & t_{a,b}(i,j) \\
\hline
\hline
    (0,0,1) &  \mathds{1} & 1 &  1 \\
    (0,0,q) &  \mathds{1} & 1 &  q^{-ij} \\
    (0,0,q^2) &  \mathds{1} &  1 & q^{-2ij} \\ 
       \hline
    (1,2,\zeta^2) &  g \text{ when } i+j \ge 3 & q^{2k} \text{ when } i+j \ge 3 &\zeta^{-2ij} \\
       (1,2,\zeta^8) &  g \text{ when } i+j \ge 3&q^{2k} \text{ when } i+j \ge 3  & \zeta^{-8ij}\\
       (1,2,\zeta^{14}) & g \text{ when } i+j \ge 3& q^{2k} \text{ when } i+j \ge 3 & \zeta^{-14ij}\\
       \hline
       (2,1,\zeta^{-2} ) & g^2 \text{ when } i+j \ge 3 &q^{k} \text{ when } i+j \ge 3  & \zeta^{2ij} \\
     (2,1, \zeta^{-8} ) &  g^2 \text{ when } i+j \ge 3  & q^{k} \text{ when } i+j \ge 3 &\zeta^{8ij}  \\
     (2,1,\zeta^{-14} ) & g^2 \text{ when } i+j \ge 3  &q^{k} \text{ when } i+j \ge 3  & \zeta^{14ij}\\
\end{array}\]\\

\end{example}

\subsection{Zested rigidity, pivotality, and trace} 
\label{sec:zestedstructures}
Often one is interested in $G$-crossed braided categories with additional properties and structure. We show that the property of rigidity is preserved by $G$-crossed zesting (Section \ref{sec:zestedrigidity}) and that there is a natural pivotal structure on the $G$-crossed zesting of a pivotal fusion category (Section \ref{sec:zestedpivotality}). However, we find there is a mild obstruction to sphericality coming from the quantum dimensions and twists of the invertible objects used in the zesting. In contrast to braided zesting the traces of morphisms are (almost) unchanged.
              
\subsubsection{Rigidity of \texorpdfstring{$G$}{G}-crossed associative zesting}    \label{sec:zestedrigidity}
Let $(\lambda, \nu)$ be a $G$-crossed zesting datum of a $G$-crossed braided fusion category $\cC$. The dual objects and their evaluation and coevaluation morphisms in $\cC^{(\lambda,\nu)}$ can be defined in the same manner as in associative zesting \cite[Section 3.3]{DGPRZ} by replacing the half-braidings with $G$-crossed braidings. Recall that $\lambda$ is assumed to be normalized so that $\lambda(g,e) = \lambda(e,g) = \bf{1}$ and that $\cC_{pt}$ is assumed to be strictly spherical so that $\lambda(g,h)=\lambda(g,h)^{**}$ for all $g,h \in G$ per Remark \ref{rmk:technical}.

For $X_g \in \cC_g$ with dual object $X_g^* \in \cC_{g^{-1}}$, the zested dual $\overline{X_g} \in \cC_{g^{-1}}$ is given by

\begin{align}
\overline{X_g}:= X_g^{*} \otimes \lambda(g,g^{-1})^*.
\end{align}

Using the isomorphisms $\nu^g:=\nu_{g,g^{-1},g}: \lambda(g,g^{-1})^g \to \lambda(g^{-1},g)$ for $g \in G$, suitable evaluation morphisms $\phi_{X_g}: \overline{X_g} \overset{\lambda}{\otimes} X_g \to \bf{1}$ and coevaluation morphisms $\rho_{X_g}: \textbf{1} \to X_g \overset{\lambda}{\otimes}  \overline{X_g}$ satisfying the rigidity axioms are given by the following diagrams.

\begin{eqnarray} 
\label{eval}
\phi_{X_g}  =& \begin{tikzpicture}[scale=1.5, line width=1,baseline=-10]
\draw[looseness=1.5] (0,0) node[above] {$X_g^{*}$} to [out=-90, in=-90] (2,0) node[above] {$\phantom{\overline{X_g}}X_g\phantom{\overline{X_g}}$};
\draw[looseness=1.5, white, line width=10] (1,0) to [out=-90, in=-90] (3,0);
\draw (1,0) node[above] {$\lambda(g,g^{-1})^{*}$};
\draw[looseness=1.5] (1,0)  to [out=-90, in=-90] (3,0);
\draw (3,0) node[above] {$\phantom{\lambda(g)^{*}}\lambda(g^{-1},g)\phantom{\lambda(g)^{*}}$};
\draw[fill=white] (2.7,-.1) rectangle node {\scriptsize $\nu^g$} (3.1, -.5);
\draw (3.1, -.2) node[right]  {$^{-1}$};
\end{tikzpicture} \\
\rho_{X_g} =& 
\begin{tikzpicture}[scale=1.5,line width=1,baseline=0]
\draw[looseness=2] (0,0) node[below] {$\phantom{X_g^{*}}X_g\phantom{X_g^{*}}$} to [out=90, in=90] (1,0) node[below] {$X_g^{*}$};
\draw (1.875,0) node[below] {$\lambda(g,g^{-1})^{*}$};
\draw[looseness=2] (2,0) to [out=90, in=90] (3,0)  ;
\draw (3.125,0) node[below] {$\phantom{\lambda(g)^{*}}\lambda(g,g^{-1})\phantom{\lambda(g)^{*}}$};
\end{tikzpicture}
\label{coeval}
\end{eqnarray}

\begin{remark}
The crossing in $\phi_{X_g}$ is opposite from the crossing used in \cite[Equation 3.7]{DGPRZ} because of the convention we have chosen for the $G$-crossed braiding in Definition \ref{def:gcrossedbraided}, see Remark \ref{rem:reverse zesting}.
\end{remark}

\subsubsection{Pivotality, trace, and sphericality of \texorpdfstring{$G$}{G}-crossed associative zesting}
\label{sec:zestedpivotality}
Without loss of generality, we may assume that $\cC$ is strict pivotal. Under zesting the pivotal structure consisting of tensor natural isomorphisms 
\[ \psi_{X_g}: X_g \to \overline{\overline{X_g}} \]
is given by the following picture
\begin{align}
    \label{defn:zestedpivotal}
\psi_{X_g}&= 
\begin{tikzpicture}[line width=1,baseline=30]
\draw (1,0) node[below] {$X_g^{**}$}--(1,3) node[above] {$X_g$};
\draw[white, line width=10, looseness=1.5] (-.5,0)--(-.5,1) to [out=90,in=90] (2.5,1) --(2.5,0);
\draw[looseness=1.5] (-.5,0) node[below] {$\lambda(g,g^{-1})$}--(-.5,1) to [out=90,in=90] (2.5,1) --(2.5,0)node[below] {$\lambda(g^{-1},g)^{*}$};
\draw[fill=white] (-.5-.33, .66) rectangle node {\small $\nu^{g^{-1}}$} (-.5+.33, 1.33);
\end{tikzpicture}.
\end{align}
\begin{remark}Once again the diagram is essentially the same as for braided zesting of pivotal fusion categories, except the crossing in $\psi_{X_g}$ is opposite from the crossing used in \cite[Equation 5.6]{DGPRZ} and the isomorphism $j$ does not appear.
\end{remark}

Now let $\dim_{\psi}(X_g)=\Tr^{\phi,\rho}(\psi_{X_g})$ be the left categorical trace of $\psi_{X_g} \in \Hom(X_g, \overline{\overline{X_g}})$ with respect to $\psi$. Then $\cC^{(\lambda,\nu)}$ will be spherical if $\dim_{\psi}(X_g)=\dim_{\psi}(\overline{X_g})$ for all $X_g \in \cC_g$ \cite[Definition 4.7.14]{EGNObook}.
We see that 

\begin{align}
\label{zesteddimobj}
\dim_{\psi}(X_g)&=
\scalebox{.7}{\begin{tikzpicture}[line width=1,baseline=30]
 \draw (1,0) node[left] {}--(1,3) node[left] {$X_g$};
    \draw[white, line width=10, looseness=1.5] (-.5,0)--(-.5,1) to [out=90,in=90] (2.5,1) --(2.5,0);
    \draw[looseness=1.5] (-.5,0) node[left] {\small $\lambda(g,g^{-1})$}--(-.5,1) to [out=90,in=90] (2.5,1) --(2.5,0)node[left] {};
    \draw[fill=white] (-.5-.33, .66) rectangle node {\small $\nu^{g^{-1}}$} (-.5+.33, 1.33);
    \begin{scope}[xshift=1cm,yshift=3cm,scale=2]
    \zestedcap
    \end{scope}
    \draw[looseness=1.5] (1,0) to [out=-90,in=-90](3,0);
    \draw[looseness=1.5] (-.5,0) to [out=-90,in=-90](4,0);
    \draw[looseness=1.5, white, line width=10] (2.5,0) to [out=-90,in=-90] (6,0) ;
    \draw[looseness=1.5] (2.5,0) to [out=-90,in=-90] (6,0) ;
    \draw[fill=white] (6-.33,-.66) rectangle node {\small $\nu^{g^{-1}}$} (6+.33,0);
    \draw (6.33, -.2) node[right]  {$^{-1}$};
    \foreach \x in {3,4,6}
    {\draw (\x,0)--(\x,3);}
    \draw (6,3) node[right]{\small $\lambda(g,g^{-1})$};
        \end{tikzpicture}}
        &= \dim(X_g)\theta_{\lambda(g,g^{-1})} 
\end{align}

and 

\begin{align}
\label{zesteddimobjdual}
\dim_{\psi}(\overline{X_g})&=
\scalebox{.7}{\begin{tikzpicture}[line width=1,baseline=30]
 \draw (1,0) node[left] {}--(1,3) node[left] {$X_g^*$};
 \draw[looseness=1.5] (1.5,0) node[right] {}--(1.5,3) to [out=90,in=90] (2.75,3)--(2.75,0) to [out=-90,in=-90] (1.5,0);
 \draw[<-, line width=.5] (2.75, 3.5)-- (3,4.25) node[above] {\small $\lambda(g,g^{-1})$};
    \draw[white, line width=10, looseness=1.5] (-.5,0)--(-.5,1) to [out=90,in=90] (2.5,1) --(2.5,0);
    \draw[looseness=1.5] (-.5,0) node[left] {\small $\lambda(g^{-1},g)$}--(-.5,1) to [out=90,in=90] (2.5,1) --(2.5,0)node[left] {};
      \draw[fill=white] (-.5-.33, .66) rectangle node {\small $\nu^g$} (-.5+.33, 1.33);
    \begin{scope}[xshift=1cm,yshift=3cm,scale=2]
    \zestedcap
    \end{scope}
    \draw[looseness=1.5] (1,0) to [out=-90,in=-90](3,0);
    \draw[looseness=1.5] (-.5,0) to [out=-90,in=-90](4,0);
    \draw[looseness=1.5, white, line width=10] (2.5,0) to [out=-90,in=-90] (6,0) ;
    \draw[looseness=1.5] (2.5,0) to [out=-90,in=-90] (6,0) ;
    \draw[fill=white] (6-.33,-.66) rectangle node {\small $\nu^g$} (6+.33,0);
    \draw (6.33, -.2) node[right]  {$^{-1}$};
    \foreach \x in {3,4,6}
    {\draw (\x,0)--(\x,3);}
    \draw (6,3) node[right]{\small $\lambda(g^{-1},g)$};
        \end{tikzpicture}}
        &= \dim(X_g^*)\theta_{\lambda(g^{-1},g)}.
\end{align}

where have used that $\cC_{pt}$ is spherical braided fusion and hence ribbon fusion to identify the twists $\theta_{\lambda(g,g^{-1})}$ and $\theta_{\lambda(g^{-1},g)}$. 

If $\cC$ is spherical then $\dim(X_g)=\dim(X_g^*)$ and we can conclude that its zesting is also spherical when 
\begin{align}
\label{sphericalcondition}
\theta_{\lambda(g,g^{-1})}=\theta_{\lambda(g^{-1},g)}
\qquad \text{ for all } g \in G.
\end{align}

\begin{remark}
Note that if $\cC$ is a unitary spherical $G$-crossed braided fusion category and $\lambda$ is a symmetric 2-cocycle then its $G$-crossed zestings are always spherical.
\end{remark}

Now one can check that the left trace of a morphism $s \in \Hom(X_g,X_g)$ becomes

\begin{align}
    \Tr^{\phi,\rho,\psi}(s) &= 
    \scalebox{.7}{\begin{tikzpicture}[line width=1,baseline=30]
 \draw (1,0) node[left] {}--(1,3);
    \draw[white, line width=10, looseness=1.5] (-.5,0)--(-.5,1) to [out=90,in=90] (2.5,1) --(2.5,0);
    \draw[looseness=1.5] (-.5,0) node[left] {\small $\lambda(g,g^{-1})$}--(-.5,1) to [out=90,in=90] (2.5,1) --(2.5,0)node[left] {};
    \draw[fill=white] (-.5-.33, .66) rectangle node {\small $\nu^{g^{-1}}$} (-.5+.33, 1.33);
    \begin{scope}[xshift=1cm,yshift=3cm,scale=2]
    \zestedcap
    \end{scope}
    \draw[looseness=1.5] (1,0) to [out=-90,in=-90](3,0);
    \draw[looseness=1.5] (-.5,0) to [out=-90,in=-90](4,0);
    \draw[looseness=1.5, white, line width=10] (2.5,0) to [out=-90,in=-90] (6,0) ;
    \draw[looseness=1.5] (2.5,0) to [out=-90,in=-90] (6,0) ;
    \draw[fill=white] (6-.33,-.66) rectangle node {\small $\nu^{g^{-1}}$} (6+.33,0);
    \draw (6.33, -.2) node[right]  {$^{-1}$};
    \foreach \x in {3,4,6}
    {\draw (\x,0)--(\x,3);}
    \draw (6,3) node[right]{\small $\lambda(g,g^{-1})$};
\draw[fill=white] (1-.33,2.66) rectangle node {$s$} (1+.33,3.33);
        \end{tikzpicture}} 
        &= \theta_{\lambda(g,g^{-1})} \Tr(s)
\end{align}
By Equation \ref{sphericalcondition} that when the zesting is spherical,

\[ \Tr^{\phi,\rho,\psi}(s)= \theta_{\lambda(g^{-1},g)} \Tr(s)\]

The right trace is drawn analogously and is equal to the right trace when the zested category is spherical. 

\section{Applications and Examples}\label{sec:appsandexes}
This section begins with an analysis of the transformation of Tambara-Yamagami categories under $\Z/N\Z$-crossed zesting. Then in Section \ref{sec:gcrossedmodulardata} we recall the definition of the $G$-crossed modular data associated with a spherical $G$-crossed braided fusion category together with some examples. Section \ref{sec:zestedmodulardata} discusses the transformation of this data under $G$-crossed zesting. 

\begin{example}[Theorem 4.9 of \cite{Galindo2022}]
\label{ex:TY}
Let $A$ be a finite abelian group, $\chi$ a non-degenerate symmetric bicharacter on $A$, and $\tau \in \{\pm\frac{1}{\sqrt{|A|}}\}$. The Tambara-Yamagami category
$\TY(A,\chi,\tau)$ has the structure of a $\Z/2\Z$-crossed braided fusion category 
$$\TY(A,\chi,\tau)=\TY(A,\chi,\tau)_0 \oplus \TY(A,\chi,\tau)_1$$ where the simple objects in the $\Z/2\Z$-graded components  $\TY(A,\chi,\tau)_0$ and $\TY(A,\chi,\tau)_1$ are the elements of $A$ and $m$ respectively, with fusion rules $a\otimes b=ab$ and $m\otimes m=\bigoplus_{a\in A}a$.

The choices of $\Z/2\Z$-crossed braided structures on $\TY(A,\chi,\tau)$ are in correspondence with pairs $(q,\alpha)$, where $q:A \to \mathds{k}^\times$ is a quadratic form and $\alpha \in \mathds{k}^\times$ with

\begin{align*}
\chi(a,b) = \frac{q(a)q(b)}{q(ab)} \quad \text{ for all } a,b \in A &\quad \text{ and } \quad \alpha^2= \tau (\sum_{a\in A} q(a)).
\end{align*}

The strict $\Z/2\Z$-action by inversion $T_{{1}}(a)=a^{-1}$ on $\TY(A,\chi,\tau)_0$ and $T_1(m)=m$ gives a $\Z/2\Z$-crossed braided structure with $\Z/2\Z$-crossed braiding

\begin{align}
c_{a,b} = \chi(a,b) \id_{ab} \quad &\quad c_{a,m}=c_{m,a} = q(a)\id_m  \quad &\quad c_{m,m} = \alpha (\bigoplus_{a \in A} q(a)^{-1} \id_a).  
\end{align} 
 
{Fix a choice of a quadratic form $q$ and $\alpha$ as above. Let us determine the possible (normalized) $2$-cocycles $\lambda\in Z^2(\Z/2\Z,A)$. If  $\lambda(1,1)=h$ then the $2$-cocycle condition implies $$\lambda(1,1)^1\otimes\lambda(0,1)=h^{-1}=\lambda(1,1)\otimes\lambda(1,0)=h.$$ This implies that $h^2=e$. Fix such an $h\in A$ (necessarily $h=e$ if $|A|$ is odd). Since $m\otimes h=m$ the zested fusion rules are unchanged. In particular, the zested category will be a Tambara-Yamagami category again, possibly with different $\Z/2\Z$-action maps $\mu$ and $\gamma$ and different $\Z/2\Z$-braiding.  
Notice that we may obtain a category with a non-strict action, but by \cite{Galindo2022} it will be equivalent to a category with a strict action as described above.

Next we  determine the possible values of the 3-cochain $\nu$, which is trivial except for $\nu_{1,1,1}=x\cdot \id_h$. Taking $g_i=1\in\Z/2\Z$ for $i=1,\ldots,4$ in (\ref{eq:associativeGzesting}) we find:
$x^2=\chi(h,h)=\frac{q(h)^2}{q(\mathbf{1})}=q(h)^2$.  So $x=\pm q(h)$.  If $h=e$ then $q(h)=1$, while if $h$ has order 2 then  $q(h)^2=\pm 1$.  Equipped with these choices we may employ  the results of section \ref{cyclic zesting} to construct the data.  Since we know that the result is again a Tambara-Yamagami category we leave this to the reader.} 
\end{example}

\begin{remark} In the previous example $A$ can be any finite abelian group. Explicit formulas for the skeletal data of $\TY(A,\chi, \tau)$ as a $\Z/2\Z$-crossed braided fusion category with action by inversion on $\TY(A,\chi, \tau)_0$ in the case where $A = (\Z/N\Z)^p$ for $N$ odd and $\text{gcd}(p,N)=1$ in terms of $\chi(a,b)=e^{2\pi iabp/N}$ can be found in \cite[Section XI.G.2]{BBCW}.
\end{remark}

\subsection{\texorpdfstring{$G$}{G}-crossed modular data 
}
\label{sec:gcrossedmodulardata}
Recall the modular data\footnote{We used the convention $S_{X,Y} \sim \Tr(c_{Y^*,X} \circ c_{X,Y^*})$ in \cite{DGPRZ}.} of a spherical braided (aka ribbon) fusion category consisting of the $S$- and $T$-matrices 
\begin{align}
\label{eq:modulardata}
S_{X,Y} = \frac{1}{D} \Tr(c_{Y, X^*} \circ c_{X^*,Y})= \frac{1}{D} & \begin{tikzpicture}[line width=1,scale=.5,baseline=.75cm]
\draw (0,0) \br (1,1.5);
\draw[white, line width=10] (1,0) \br (0, 1.5);
\draw (1,0) \br (0, 1.5);
\begin{scope}[yshift=1.5cm]
\draw (0,0) \br (1,1.5);
\draw[white, line width=10] (1,0) \br (0, 1.5);
\draw (1,0) \br (0, 1.5);
\end{scope}
\draw[looseness=2] (1,3) to [out=90, in =90] (2,3)--(2,0) to [out=-90, in=-90] (1,0);
\draw[looseness=2] (0,3) to [out=90, in =90] (-1,3)--(-1,0) to [out=-90, in=-90] (0,0);
\draw[->] (-1,1.5)node[left] {$X$} --(-1,1.25) ; 
\draw[->] (2,1.5) node[right] {$Y$}; 
\end{tikzpicture}  \\ T_{X,Y} = \frac{\delta_{X,Y}}{\dim(X)} \Tr(c_{X,X}) =  \frac{\delta_{X,Y}}{\dim(X)}    &\begin{tikzpicture}[line width=1,scale=.5,baseline=.375cm]
\draw (0,0) \br (1,1.5);
\draw[white, line width=10] (1,0) \br (0, 1.5);
\draw (1,0) \br (0, 1.5);
\draw[looseness=2] (1,1.5) to [out=90, in =90] (2,1.5)--(2,0) to [out=-90, in=-90] (1,0);
\draw[looseness=2] (0,1.5) to [out=90, in =90] (-1,1.5)--(-1,0) to [out=-90, in=-90] (0,0);
\draw[<-] (-1,.75) node[left] {$X$};
\end{tikzpicture} \quad =\theta_X \delta_{X,Y}
\end{align}
for $X,Y \in \Irr(\cC)$, where $D^2 = \sum_{X \in \Irr(\cC)} \dim(X)^2$ and $\theta_X$ is the twist of $X$.

The modular data are invariants of ribbon fusion categories under ribbon autoequivalences; in particular, the diagrams for $S$ and $T$ above make sense whether your category is strict or skeletal (or neither for that matter). The $S$- and $T$- matrices yield a representation of the modular group $SL(2,\mathbb{Z})$ and possess other symmetries which further constrain the form they can take. This makes the modular data an important tool for classifying modular fusion categories and handling examples, especially at low rank \cite{RSW}. \\

In a spherical $G$-crossed braided fusion category $\cC$ one can compute the trace of $c_{Y_h,(X_g^*)^h}\circ c_{X_g^*,Y_h}$ whenever it is well-defined, i.e. for those $X_g \in \cC_g$, $Y_h \in \cC_h$ with isomorphisms $X_g^h \cong X_g$ and $Y_h^g \cong Y_h$ \cite{Kirillov}. When $\cC$ is skeletal these isomorphisms become equalities and one can define a matrix $S^{(g,h)}_{X_g,Y_h}$ by labeling the usual Hopf link diagram of Equation \ref{eq:modulardata} with pairs $X_g$ and $Y_h$ satisfying $X_g^h=X_g$ and $Y_h^g=Y_g$, see \cite[Section VII]{BBCW}. In general, the isomorphisms $X_g^h \cong X_g$ and $Y_h^g \cong Y_h$ are necessary to define $S^{(g,h)}$, as originally considered in \cite{Kirillov} and depicted below in Equation \ref{eq:GcrossedS}. The matrices $S^{(g,h)}$ and $S= \oplus S^{(g,h)}$ fail to be invariant under the $G$-crossed braided autoequivalence of Definition \ref{def:G-equiv}, but when $\cC_e$ is modular $S$ still determines the fusion rules of $\cC$ through a $G$-crossed Verlinde formula. 

Similarly, when $X^g_g \cong X_g$ one can compute the ``twist" $\theta_{X_g}$. It is possible to arrange the twists into a $G$-crossed $T$-matrix so that $S$ and $T$ define a projective representation of the modular group \cite[Section VII]{BBCW}.  

Following \cite{Kirillov} and \cite{BBCW} we recall the definition of the $G$-crossed $S$-matrix and twists.

\begin{definition}
\label{def:gcrossedmodulardata}
Let $\cC$ be a spherical $G$-crossed braided fusion category with a fixed set of representatives $X_g$ of $\Irr(\cC_g)$ for all $g \in G$ and a fixed basis $\phi$ of $\Hom(X_g^h,X_g)$ for all $X_g \in \Irr(\cC_g)$ and $h \in G$. Let $\mathcal{V}_{(g,h)}:=\oplus_{X_g \in \Irr(\cC_g)} \Hom(X_g^h,X_g)$ and define $S^{(g,h)}: \mathcal{V}_{(h,g^{-1})} \to \mathcal{V}_{(g,h)}$ 
to be the linear transformation whose matrix entries are given by\footnote{In Kirillov's notation $\oplus_{X_g \in \Irr(\cC_g)} \Hom(X_g^{h^{-1}},X_g) = \mathcal{V}_{h^{-1},g}$ \cite{Kirillov}. } 

\begin{align}
\label{eq:GcrossedS}
      S^{(g,h)}_{X_g,Y_h} & =\frac{1}{D_e}\,
      \begin{tikzpicture}[line width=1,scale=.75,baseline=.75cm]
\draw (0,-1)--(0,0) \br (1,1.5);
\draw[white, line width=10] (1,0) \br (0, 1.5);
\draw (1,-1)--(1,0) \br (0, 1.5);
\begin{scope}[yshift=1.5cm]
\draw (0,0) \br (1,1.5);
\draw[white, line width=10] (1,0) \br (0, 1.5);
\draw (1,0) \br (0, 1.5);
\end{scope}
\begin{scope}[yshift=-.35cm]
\draw[fill=white] (-.35,-.35) rectangle node {$\phi$} (.35,.35);
\draw (.2, .25) node[right] {$^*$};
\begin{scope}[xshift=1cm]
\draw[fill=white] (-.35,-.35) rectangle node {$\psi$} (.35,.35);
\end{scope}
\end{scope}
\draw[looseness=2] (1,3) to [out=90, in =90] (2,3)--(2,-1) to [out=-90, in=-90] (1,-1);
\draw[looseness=2] (0,3) to [out=90, in =90] (-1,3)--(-1,-1) to [out=-90, in=-90] (0,-1);
\draw[<-] (-1,1) node[left] {$X_g$} --(-1,1.15); 
\draw[->] (2,1)node[right] {$Y_h$}; 
\end{tikzpicture} 
    \end{align}
    where $X_g \in \Irr(\cC_g)$, $Y_h \in \Irr(\cC_h)$ with $\phi: X_g^h \cong X_g$ and $\psi: Y_h^{g^{-1}} \cong Y_h$.
    
Let $\mathcal{V} = \bigoplus_{\{(g,h) | gh=hg\}} \mathcal{V}_{g,h}$.  Then the $G$-crossed $S$-matrix $S:\mathcal{V} \to \mathcal{V}$ is
\begin{align}
S &= \bigoplus_{\{(g,h) | gh=hg\}} S^{(g,h)}.
\end{align}
    
For $X_g \in \Irr(\cC_g)$ and $\pi:X_g^g \cong X_g$ the twist $\theta_{X_g}$ is given by 
    \begin{align}
    \label{eq:Gcrossedtwist}
      \theta_{X_g} & = \frac{1}{d_{X_g}}\,    \begin{tikzpicture}[line width=1,scale=.75,baseline=0]
\draw (0,-1)--(0,0) \br (1,1.5);
\draw[white, line width=10] (1,0) \br (0, 1.5);
\draw (1,-1)--(1,0) \br (0, 1.5);
\begin{scope}[yshift=-.35cm]
\draw (0,0)--(0,.35);
\begin{scope}[xshift=1cm]
\draw[fill=white] (-.35,-.35) rectangle node {$\pi$} (.35,.35);
\end{scope}
\end{scope}
\draw[looseness=2] (1,1.5) to [out=90, in =90] (2,1.5)--(2,-1) to [out=-90, in=-90] (1,-1);
\draw[looseness=2] (0,1.5) to [out=90, in =90] (-1,1.5)--(-1,-1) to [out=-90, in=-90] (0,-1);
\draw[->] (-1,.25) node[left] {$X_g$};
\end{tikzpicture}.
\end{align}
\end{definition}

\begin{remark}
There are a number of different conventions for $G$-crossed modular data used here and elsewhere in the literature. These are mostly cosmetic: orientations of diagrams, whether it is the overstrand or understand of a braid that gets acted on in a $G$-crossed braiding, which component is labeled by $g$ and which by $h$, etc. There also seems to be a split between which of $\oplus_{X_g \in \Irr(\cC_g)} \Hom(X_g^{h^{-1}},X_g)$ and $\oplus_{Y_h \in \Irr(\cC_h)} \Hom(Y_h^g,Y_h)$ is viewed as the source and which as the target of $S^{(g,h)}$. In Definition \ref{def:gcrossedmodulardata} we have largely followed the conventions in \cite{BBCW}.
\end{remark}
In the following examples $G$ is cyclic and we arrange the blocks of the $S$-matrix with respect to the lexicographic order on $G \cong \Z/N\Z \times \Z/N\Z$ so that the $(g,h)$-row and $(h,g^{-1})$-column of $S$ is $S^{(g,h)}$.

\begin{example}
\label{ex:Z3ST}
As a special case of the $G$-crossed braided fusion categories $\Vect_G^{\omega}$ in Example \ref{ex:vecgomega}, when the group $G=A$ is abelian the categories are parametrized by $\Vect_A^{q}$ where $q$ is a quadratic form. Take for example $A=\Z/3\Z$. The $\Z/3\Z$-crossed $S$-matrix and twists for $\Vect^{q}_{\Z/3\Z}$ is given by

\begin{align*}S_{\Vect^{q}_{\Z/3\Z}} =\left(
\begin{array}{ccccccccc}
 1 & 0 & 0 & 0 & 0 & 0 & 0 & 0 & 0 \\
 0 & 0 & 0 & 1 & 0 & 0 & 0 & 0 & 0 \\
 0 & 0 & 0 & 0 & 0 & 0 & 1 & 0 & 0 \\
 0 & 0 & 1 & 0 & 0 & 0 & 0 & 0 & 0 \\
 0 & 0 & 0 & 0 & 0 & q^2 & 0 & 0 & 0 \\
 0 & 0 & 0 & 0 & 0 & 0 & 0 & 0 & q \\
 0 & 1 & 0 & 0 & 0 & 0 & 0 & 0 & 0 \\
 0 & 0 & 0 & 0 & q & 0 & 0 & 0 & 0 \\
 0 & 0 & 0 & 0 & 0 & 0 & 0 & q^2 & 0 \\
\end{array}\right) 
\quad
\theta_0 = 1, \theta_1=\theta_2=q
\end{align*}

Different choices of $q$ yield inequivalent braided fusion categories: when $q \equiv 1$ $\Vect_{\Z/3\Z}$ is a symmetric fusion category, while when $q(a)= e^{2\pi ia/3}$ for $a \in {1,2}$ it is modular. On the other hand, $\Vect_{\Z/3\Z} \simeq \Vect^{q}_{\Z/3\Z}$ as $\Z/3\Z$-crossed braided fusion categories; one can check that they generate equivalent $\Z/3\Z$-equivariantizations. Observe however that the $\Z/3\Z$-crossed $S$-matrices for different values of $q$ are not equal up to a relabeling. 
\end{example}

\begin{example}
\label{ex:vecgomegaST}
Consider $\Vect_G^{\omega}$ as a $G$-crossed braided fusion category as in Example \ref{ex:vecgomega} where $G$ acts by conjugation and the $G$-crossed braiding $c_{g,h}: g \otimes h \to h \otimes h^{-1}gh$ is the identity. The block structures of the $G$-crossed $S$- and $T$-matrices do not admit a uniform description since they depend on the multiplication in $G$, but for commuting pairs $(g,h)$ one has

\begin{align*}
    S^{(g,h)}_{g,h} = \frac{1}{\sqrt{|G|}}, \quad &\quad  \theta_g= 1. 
\end{align*}
\end{example}

\begin{example}
\label{ex:TYmodulardata}
Consider the Tambara-Yamagami category $\TY(A,\chi,\tau)$ from Example \ref{ex:TY} with the strict $\langle g \rangle = \mathbb{Z}/2\mathbb{Z}$-action that sends $T_g(a)=-a$ and $T_g(m)=m$. We have

\begin{align*} S = \left(\begin{array}{c|c|c|c}
S^{(0,0)} & \Large{0} & \Large{0} & \Large{0} \\
\hline
 \Large{0} & \Large{0} & S^{(g,0)} & \Large{0} \\
\hline
 \Large{0} & S^{(0,g)} & \Large{0} & \Large{0} \\
 \hline
  \Large{0} & \Large{0} & \Large{0} & S^{(g,g)}
\end{array}\right)
,\quad
\theta_{a} = \frac{\chi(a,a)}{d_a}, \quad & \quad\theta_{m}= \frac{\alpha}{\sqrt{|A|}} \sum_{a \in A}q(a)^{-1}d_a
\end{align*}
where 
$S^{(0,0)}_{a,b}=\frac{d_ad_b}{\sqrt{|A|}}\chi(-a,b)\chi(b,-a)$, $S^{(0,g)}_{a,m}=S^{(g,0)}_{m,a}=\frac{q(a)^2d_m}{\sqrt{|A|}}$ for $a=-a$, and $S^{(g,g)}_{m,m}=\frac{\alpha^2}{\sqrt{|A|}}(\sum_{a \in A}q(a)^{-1})^2d_a$.
\end{example}

\subsection{Zested \texorpdfstring{$G$}{G}-crossed modular data}
\label{sec:zestedmodulardata}

In \cite[Theorem 5.7]{DGPRZ} we gave formulas for the modular data of a zested ribbon fusion category $\cC^{(\lambda, \nu, t,j,f)}$ in terms of the original modular data of $\cC$ and the ribbon zesting data $(\lambda, \nu, t, j, f)$. Naturally, one wants analogous formulas for the zested $G$-crossed modular data. However, it ends up being subtle to compute the $G$-crossed $S$-matrix of the zested category, as the following discussion shows.

\subsubsection{Zesting does not preserve the fixed points of the \texorpdfstring{$G$}{G}-action on \texorpdfstring{$\Irr(\cC)$}{Irr(C)}}

Consider the formula for $S^{(g,h)}_{X_g,Y_h}$ in Definition \ref{def:gcrossedmodulardata} above for $X_g$ and $Y_h$ with $X_g^{h} \cong X_g$ and $Y_h^g \cong Y_h$. In $\cC^{(\lambda,\nu)}$ it is no longer necessarily true that
\begin{align*} & \quad 
X_g^{h^{\lambda}} = X_g^{h} \otimes  \lambda(h,g) \otimes \lambda(g,h)^*
\text{ and } \quad
(Y_h)^{g^{\lambda}} = (Y_h)^{g} \otimes \lambda(g,h) \otimes \lambda(h,g)^* 
 \end{align*}
 are isomorphic to $X_g$ and $Y_h$ respectively for commuting $g,h$ unless there are isomorphisms $\lambda(g,h) \cong \lambda(h,g)$. In other words, unlike in the braided case it is not clear how to meaningfully compare entries of the $G$-crossed $S$-matrix before and after zesting: with respect to a fixed ordering of $G$ and $\Irr(\cC)$, the $G$-crossed $S$-matrix of $\cC^{(\lambda,\nu)}$ from Equation \ref{eq:GcrossedS} may not even have the same block structure as that of $\cC$ since there is no general relationship between $\Hom_{\cC}(X_g^{h}, X_g)$ and $\Hom_{\cC^{(\lambda,\nu)}}(X_g^{h^\lambda}, X_g)$ and a set of representative isomorphisms $X_g^{h} \cong X_g$ of the basis elements of $\mathcal{V}_{(g,h)}$ does not necessarily extend to a basis of $\mathcal{V}_{(g,h)}^{\lambda,\nu}:= \bigoplus_{X_g \in \Irr(\cC_g)} \Hom_{\cC^{(\lambda,\nu)}}(X_g^{h^\lambda},X_g)$.
 
\subsubsection{A zested \texorpdfstring{$G$}{G}-crossed \texorpdfstring{$S$}{S}-matrix and twists} 

However, one still has $(X_g^*)^{h^{\lambda}} \otimes Y_h^{(g^{-1})^{\lambda}} \cong X_g^* \otimes Y_h$. Define 
$\tau_{g,h}: (X_g^*)^{h^{\lambda}} \otimes Y_h^{(g^{-1})^{\lambda}} \to X_g^* \otimes Y_h$ by

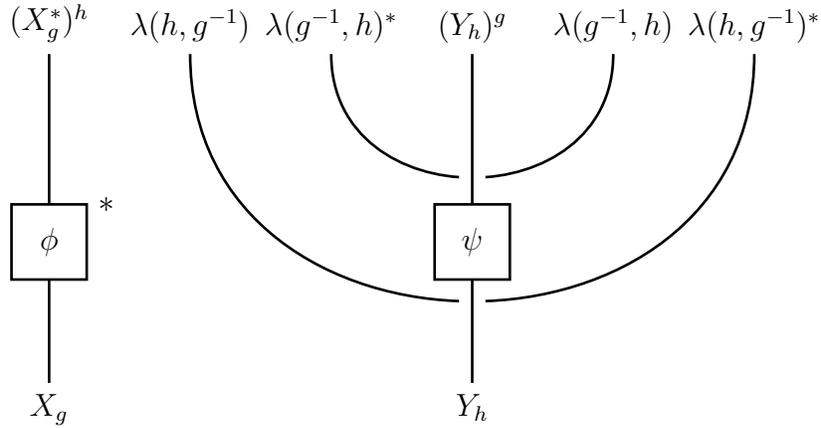
\begin{align*}
\begin{tikzpicture}[scale=1.25,line width=1]
\draw (0,-1.5) node[below] {$X_g$}-- (0,2) node[above] {$(X_g^*)^h$};
\draw[fill=white] (-.4,-.4 ) rectangle node {$\phi$}(.4,.4) node[right] {$*$} ;
\draw[looseness=1.5] (3,2) node[above] {$\lambda(g^{-1},h)^*$} to [out=-90,in=-90] (6,2) node[above] {$\lambda(g^{-1},h)$};
\draw[looseness=1.5] (1.5,2) node[above] {$\lambda(h,g^{-1})$} to [out=-90,in=-90] (7.5,2) node[above] {$\lambda(h,g^{-1})^*$};
\draw[white, line width=10] (4.5,-1.5) -- (4.5,2);
\draw (4.5,-1.5) node[below] {$Y_h$}-- (4.5,2) node[above] {$(Y_h)^g$};
\draw[fill=white] (4.1,-.4 ) rectangle node {$\psi$}(4.9,.4) ;
\end{tikzpicture}
\end{align*}
where $\phi:X_g^{h} \to X_g$ and $\psi:Y_h^g \to Y_h$ are as in Definition \ref{def:gcrossedmodulardata}. 

Then $(\tilde{S}^{(\lambda,\nu)})^{(g,h)}_{X_g,Y_h}:= \Tr^{(\lambda,\nu)}\left (\tau_{g,h} \circ c_{Y_h, (X_g^*)^{h^{\lambda}}}\circ c_{X_g^*, Y_h} \right)$ is well-defined and satisfies
\begin{align}
(\tilde{S}^{\lambda,\nu})^{(g,h)}_{X_g,Y_h}= \frac{1}{D_e} &  
\begin{tikzpicture}[line width=1,scale=.5,baseline=0]
\draw (11,-4.5)--(11,3)  to [out=90, in=90] (14,3)--(14,-4.5);
\draw (1,-3) \br (2,1.5) -- (2,3) ;
\draw (0,-3) \br (1,1.5);
\draw[looseness=1.5] (2, -3) \br (3,1.5) to [out=90, in=90] (4,1.5)--(4,-3);
\draw[white, line width=10] (3,-3) \br (0, 1.5);
\draw (3,-3) \br (0, 1.5);
\begin{scope}[yshift=1.5cm]
\draw (0,0) \br (1,1.5);
\draw[white, line width=10] (1,0) \br (0, 1.5);
\draw (1,0) \br (0, 1.5);
\end{scope}
\draw[looseness=1.5] (5,-3) to [out=90, in=90] (6,-3);
\foreach \x in {0,1,2,3,4,5,6,8,9,10}{
\draw (\x,-3)-- (\x,-4.5);}
\draw[white, line width=10, looseness=.5] (-1,-4.5) to [out=90, in=90] (7,-4.5);
\draw[looseness=.5] (-1,-4.5) to [out=90, in=90] (7,-4.5);
\draw (0,3) to [out=90,in=90] (10,3)--(10,-3);
\draw (1,3) to [out=90,in=90] (9,3)--(9,-3);
\draw (2,3) to [out=90,in=90] (8,3)--(8,-3);
\draw (2,-4.5) to [out=-90,in=-90] (4,-4.5);
\draw (1,-4.5) to [out=-90,in=-90] (5,-4.5);
\draw[white, line width=10] (3,-4.5) to [out=-90, in=-90] (9,-4.5);
\draw (3,-4.5) to [out=-90, in=-90] (9,-4.5);
\draw (0,-4.5)  to [out=-90, in=-90] (10,-4.5);
\draw (-1,-4.5)  to [out=-90, in=-90] (11,-4.5);
\draw (6,-4.5) to [out=-90, in=-90] (8,-4.5);
\draw[white, line width=10] (7,-4.5) to [out=-90, in=-90] (14,-4.5);
\draw(7,-4.5) to [out=-90, in=-90] (14,-4.5);
\draw[fill=white] (13.5,-4.5) rectangle node {\small $\nu^k$} (14.5,-3.5) node[right] {$^{-1}$};
\draw[fill=white] (-1.5,-5.5) rectangle node {\small $\nu^k$} (-.5,-4.5);
\draw [fill=white] (-.33, -3) rectangle node {\small$\phi$} (.33,-2.33);
\draw (0.15,-2.33) node[right] {$^*$};
\draw [fill=white] (3-.33, -3) rectangle node {\small $\psi$} (3.33,-2.33);
\end{tikzpicture} 
=& 
S^{(g,h)}_{X_g,Y_h} \theta_{\lambda(k^{-1},k)}
\end{align}

where $k = h^{-1}g$ and we have abbreviated the isomorphisms $\nu^k:= \nu_{k, k^{-1},k}$ to save space.

The situation is not quite so bad with the twists, and one can check that there is a basis on $\mathcal{V}_{g,g}^{(\lambda,\nu)}$ given by $\pi \otimes \text{eval}_{\lambda(g,g)}$ 
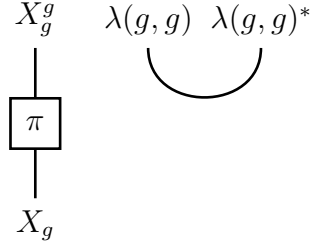
\begin{align*}
\begin{tikzpicture}[scale=1, line width=1]
\onebox{X_g}{X_g^{g}}{\pi}
\draw[looseness=1.5] (1.5,2) node[above] {$\lambda(g,g)$} to [out=-90,in=-90] (3,2) node[above] {$\lambda(g,g)^*$};
\end{tikzpicture}
\end{align*}
that one can use to directly compare the twists before and after zesting. 

\begin{align}
\tilde{\theta}^{\lambda,\nu}_{X_g}= \frac{1}{\dim(X_g)}&
\begin{tikzpicture}[line width=1,scale=.5,baseline=0]
\draw (9,-1.5)--(9,3)  to [out=90, in=90] (11,3)--(11,-1.5);
\draw (2,1.5) -- (2,3) ;
\begin{scope}[xshift=1cm]
\onebox{}{}{\pi}
\draw[looseness=1.5] (1,1.5) node[above] {} to [out=-90,in=-90] (2,1.5) node[above] {};
\draw[looseness=1.5] (2,1.5) node[above] {} to [out=90,in=90] (3,1.5) node[above] {};
    \end{scope}
\begin{scope}[yshift=1.5cm]
\draw (0,-1.5)--(0,0) \br (1,1.5);
\draw[white, line width=10] (1,0) \br (0, 1.5);
\draw (1,0) \br (0, 1.5);
\end{scope}
\draw[looseness=1.5] (5,-3) to [out=90, in=90] (6,-3);
\foreach \x in {0,1,4}{
\draw (\x,0)-- (\x,-1.5);}
\draw (4,0)--(4,1.5);
\draw[white, line width=10, looseness=.9] (-1,-1.5) to [out=90, in=90] (5,-1.5);
\draw[looseness=.9] (-1,-1.5) to [out=90, in=90] (5,-1.5);
\draw (0,3) to [out=90,in=90] (8,3)--(8,-1.5);
\draw (1,3) to [out=90,in=90] (7,3)--(7,-1.5);
\draw (2,3) to [out=90,in=90] (6,3)--(6,-1.5);
\draw[white, line width=10] (4,-1.5) to [out=-90, in=-90] (9,-1.5);
\draw (4,-1.5) to [out=-90, in=-90] (6,-1.5);
\draw (0,-1.5)  to [out=-90, in=-90] (8,-1.5);
\draw (1,-1.5) to [out=-90, in=-90] (7,-1.5);
\draw (-1,-1.5)  to [out=-90, in=-90] (9,-1.5);
\draw[white, line width=10] (5,-1.5) to [out=-90, in=-90] (11,-1.5);
\draw(5,-1.5) to [out=-90, in=-90] (11,-1.5);
\draw[fill=white] (9.4,-3.6) rectangle node {\tiny $g^{-2}$} (10.6,-2.4) node[right] {$^{-1}$};
\draw[fill=white] (1.9,-0.6) rectangle node {\tiny $g^{-2}$} (3.1,.6);
\end{tikzpicture} 
= \theta_{X_g} \theta_{\lambda(g^2,g^{-2})}
\end{align}

One could think of $\tilde{S}$ and $\tilde{\theta}$ as ``zested $G$-crossed modular data" for $\cC^{(\lambda,\nu)}$ and compare it directly to that of $\cC$, although one must keep in mind that neither are invariants and that it may differ from the $G$-crossed modular data of $\cC^{(\lambda,\nu)}$ computed directly from Definition \ref{def:gcrossedmodulardata}. 

\bibliography{refs}{}
\bibliographystyle{plain}
\end{document}